\documentclass[12pt]{amsart}
\usepackage{amssymb,amsmath,amsfonts,latexsym}
\usepackage{bm}
\usepackage[all,cmtip]{xy}
\usepackage{amscd}
\usepackage{setspace}
\usepackage{mathrsfs}
\usepackage[colorlinks ,citecolor=blue,urlcolor=blue,linkcolor=blue]{hyperref}

\allowdisplaybreaks

\newcommand{\bea}{\begin{eqnarray}}
\newcommand{\eea}{\end{eqnarray}}

\newcommand{\cla}{\mathcal{A}}
\newcommand{\clb}{\mathcal{B}}

\newcommand{\cld}{\mathcal{D}}
\newcommand{\cle}{\mathcal{E}}
\newcommand{\clf}{\mathcal{F}}

\newcommand{\clh}{\mathcal{H}}
\newcommand{\clk}{\mathcal{K}}
\newcommand{\cll}{\mathcal{L}}
\newcommand{\clm}{\mathcal{M}}
\newcommand{\cln}{\mathcal{N}}

\newcommand{\clr}{\mathcal{R}}

\def\textmatrix#1&#2\\#3&#4\\{\bigl({#1 \atop #3}\ {#2 \atop #4}\bigr)}
\def\dispmatrix#1&#2\\#3&#4\\{\left({#1 \atop #3}\ {#2 \atop #4}\right)}
\newcommand{\be}{\begin{equation}}
\newcommand{\ee}{\end{equation}}
\newcommand{\ben}{\begin{eqnarray*}}
\newcommand{\een}{\end{eqnarray*}}

\newcommand{\NI}{\noindent}

\newcommand{\bi}{\begin{itemize}}
\newcommand{\ei}{\end{itemize}}

\newcommand{\C}{\mathbb{C}}

\newcommand{\D}{\mathbb{D}}
\newcommand{\Z}{\mathbb{Z}}
\newcommand{\T}{\mathbb{T}}
\newcommand{\E}{\mathcal{E}} 
\newcommand{\F}{\mathcal{F}}

\newcommand{\he}{{H}^2_\mathcal{E}(\mathbb{D})}

\newcommand{\dt}{\mathcal{D}}

\theoremstyle{definition}

\theoremstyle{plain}

\newtheorem{thm}{Theorem}[section]
\newtheorem{cor}[thm]{Corollary}
\newtheorem{lem}[thm]{Lemma}
\newtheorem{prop}[thm]{Proposition}
\theoremstyle{definition}
\newtheorem{defn}[thm]{Definition}
\newtheorem{rem}[thm]{Remark}
\newtheorem{ex}[thm]{Example}

\numberwithin{equation}{section}

\let\phi=\varphi

\begin{document}

\title[Characteristic function of a power partial isometry]{Characteristic function of a power partial isometry}
	
\author[Babbar]{Kritika Babbar}
\address{Indian Institute of Technology Roorkee, Department of Mathematics,
		Roorkee-247 667, Uttarakhand,  India}
\email{kritika@ma.iitr.ac.in, kritikababbariitr@gmail.com}

\author[Maji]{Amit Maji}
\address{Indian Institute of Technology Roorkee, Department of Mathematics,
		Roorkee-247 667, Uttarakhand,  India}
\email{amit.maji@ma.iitr.ac.in, amit.iitm07@gmail.com ({Corresponding author)}}




\begin{abstract}
The celebrated Sz.-Nagy-Foia\c{s} model theory says that there is a bijection between the class of purely contractive analytic functions and the class of completely non-unitary (c.n.u.) contractions modulo unitary equivalence. In this paper we provide a complete classification of the purely contractive analytic functions such that the associated contraction is a c.n.u.\ power partial isometry. As an application of our findings, we determine a class of contractive polynomials such that the associated c.n.u.\ contraction is of the explicit diagonal form $S \oplus N \oplus C$, where $S$ and $C^*$ are unilateral shifts and $N$ is nilpotent. Finally, we obtain a characterization of operator-valued symbols for which the corresponding Toeplitz operator on vector-valued Hardy space is a partial isometry.
\end{abstract}
\maketitle

\section{Introduction}
One of the fundamental problems in operator theory is to find a complete unitary invariant of a bounded linear operator on a separable Hilbert space. In this context, the characteristic function of a contraction on a Hilbert space plays a vital role and acts as a bridge between operator theory and function theory. 
Since the classification problem of any bounded linear operator is generally challenging, research on this problem focuses on examining specific classes of tractable operators for which one can find nice and useful (unitary) invariants. The primary goal of this paper is to provide a comprehensive classification of purely contractive analytic functions (in particular, polynomial functions) such that the associated contraction is a completely non-unitary (c.n.u.\ in short) power partial isometry, and also to address the problem of  characterizing operator-valued symbols for which the Toeplitz operator on vector-valued Hardy space is a partial isometry.

For a single contraction on a Hilbert space, the characteristic function has a long-standing tradition.\ It has widespread applications across various disciplines like transfer function theory, perturbation theory, control theory, stability theory, and network realizability theory and so on (cf.\ \cite{Bercovici}, \cite{GKrein}, \cite{livsic2}, \cite{sznagy}, \cite{Helton}). The notion of characteristic function was first introduced by Liv\v{s}ic \cite{livsic} and major development has been done by Sz.-Nagy and Foia\c{s} \cite{sznagy} in their dilation and model theory. Thereafter, a lot of research (in particular, constant and polynomial characteristic functions) has been done by Wu \cite{PYW}, Bagchi and Misra \cite{GM}, Foia\c{s} and Sarkar \cite{foias}, Foia\c{s}, Pearcy and Sarkar \cite{foias2}. Characteristic function serves as a complete unitary invariant for c.n.u.\ contractions. More precisely, two c.n.u.\ contractions are unitarily equivalent if and only if their characteristic functions coincide (see Sz.-Nagy and Foia\c{s} \cite{sznagy}). On the other hand, an important class of contractions is power partial isometries, that is, operators whose every positive power is also a partial isometry. The decomposition of power partial isometries was first initiated by Halmos and Wallen in \cite{halmospowers}, where they demonstrated that each power partial isometry is a direct sum of a unitary operator, a unilateral shift, a backward shift, and truncated shifts. Clearly, the c.n.u.\ component of a power partial isometry is the combining components of unilateral shift, backward shift, and truncated shifts. In this paper, we obtain the characteristic function of a power partial isometry which is a purely contractive analytic function with partially isometric coefficients. Thus the natural question arises: 

\textsf{Does a purely contractive analytic function with partially isometric coefficients generate a power partial isometry?}

In fact, we present one of the main results (see Section \ref{sec3} below):

\begin{thm}
Let $\Theta : \D \rightarrow \clb(\E,\E_*)$ be a purely contractive analytic function such that 
\[
\Theta( z)=\sum_{m=1}^\infty  \theta_mz^m, \quad (z\in \D)
\]
where $\theta_m \in \mathcal{B}(\E,\E_*)$ are partial isometries for all $m \geq 1$. Then there exist a Hilbert space 
\[
\clh=\left\{\left(I-T_\Theta T_\Theta^* \right)f\oplus \Delta_\Theta g: f \in H^2_{\E_*}(\D), g \in [\he]^\perp\right\},
\]
where $\Delta_\Theta$ is a constant projection, and a c.n.u.\ power partial isometry $T$ on $\clh$ defined by
\[
T^*(u \oplus v) = e^{-it}(u(e^{it}) - u(0)) \oplus e^{-it} v(e^{it}), \quad (u \oplus v \in \clh)
\]
such that the characteristic function of $T$ coincides with $\Theta$.
\end{thm}

Our approach essentially builds upon the Sz.-Nagy and Foia\c{s}' construction of a functional model. Additionally, we prove that the characteristic function of a truncated shift is a monomial with a partially isometric coefficient which is a specific instance of the polynomial characteristic function. In \cite{foias}, authors established that the characteristic function of a c.n.u.\ contraction $T$ on $\clh$ is a polynomial if and only if $T$ has an upper triangular matricial form
\[
T=\begin{pmatrix}
S & * & *\\
0 & N & *\\
0 & 0 & C
\end{pmatrix}
\]
with respect to the orthogonal decomposition $\clh=\clh_s \oplus \clh_t \oplus \clh_b$, where $S, N, \text{ and  }C$ are unilateral shift, nilpotent, and backward shift, respectively. The question which follows naturally is: 
\vspace{-0.4 cm}
\begin{center}
\textsf{When will this matrix form be diagonal?} 
\end{center}
As a direct consequence of our result, we identify a particular class of polynomial characteristic functions for which this representation is diagonal and also describe the orthogonal decomposition spaces explicitly (see Theorem \ref{spaces} below).

In the next part of this article, we characterize some particular class of Toeplitz operators, namely the partially isometric Toeplitz operators with operator-valued symbol on the vector-valued Hardy space. The class of Toeplitz and analytic Toeplitz operators is one of the most important classes of tractable operators.
It is a highly active research topic with a growing list of applications and links in function theory and operator theory. The characterization of nonzero Toeplitz operators that are partial isometries was initiated by Brown and Douglas in \cite{BD-Partially}. They proved that nonzero partially isometric Toeplitz operators are of the form $T_{\phi}$ and $T^{*}_{\phi}$, where $\phi$ is an inner function. In \cite{DPS-Partially}, Deepak, Pradhan, and Sarkar generalized this result in the scalar-valued Hardy space over polydisc. The similar factorization result holds for the partially isometric Toeplitz operators with operator-valued symbols in the polydisc setting which was recently studied by Sarkar in \cite{ss} and he also posed the following question:

\begin{center}
\textsf{Characterize the class of  partially isometric symbols $\Phi \in L^\infty_{\clb(\cle)}$ such that $T_\Phi$ is a partial isometry.}
\end{center}

\NI
As a byproduct of our main result, we recognize a specific class of partially isometric symbols for which the corresponding Toeplitz operator is a partial isometry, which motivates us to settle the question posed by Sarkar in \cite{ss} under certain conditions. Indeed, we have the following result (see Theorem \ref{Characterization-Partial}):

\begin{thm}
For finite-dimensional Hilbert spaces $\cle$ and $\cle_*$, let $\Phi \in L^\infty_{\clb(\cle, \cle_*)}$ be such that $\Phi(e^{it})=\sum\limits_{m=-\infty}^\infty \hat{\Phi}(m)e^{imt}$ is a nonzero partial isometry a.e.\ on $\T$. Then $T_\Phi$ is a partially isometric Toeplitz operator if and only if the following conditions are satisfied:
\begin{enumerate}
\item $\Phi_+(e^{it})^*\Phi_+(e^{it})$ and $\Phi_-(e^{it})\Phi_-(e^{it})^*$ are  operator-valued constant functions a.e.\ on $\T$ where $\Phi_+ \in H^\infty_{\clb(\cle,\cle_*)}$ and $\Phi_-^* \in H^\infty_{\clb(\cle_*,\cle)}$ are contractive analytic functions.
\item $\hat{\Phi}(m)^*\hat{\Phi}(-n)=0_\cle$ and $\hat{\Phi}(-n)\hat{\Phi}(m)^*=0_{\cle_*}$ for all $m,n \geq 1$.
\end{enumerate}
\end{thm}

The structure of the rest of the paper is organized as follows: In Section \ref{sec2}, we set all the notations and definitions that will be used throughout the paper. Section \ref{sec3} focuses on  determining the characteristic function of a power partial isometry and characterize the class of contractive analytic functions for which the associated c.n.u.\ contraction is a power partial isometry. Section \ref{sec4} provides a characterization of partially isometric Toeplitz operators with operator-valued symbols. In Section \ref{sec5}, we illustrate some concrete examples to support our result.

\section{Preliminaries}{\label{sec2}}

This section compiles all the notations, definitions, and results that are used in this paper. It is assumed that every Hilbert space is a complex separable Hilbert space. For a Hilbert space $\clh$, $I_\clh$ and $0_\clh$  represent the identity operator and zero operator on $\clh$, respectively. If $\clh$ is clear from the context, we frequently write $I$ and $0$ without the subscript. Let $\clb(\clh)$ be the $C^*$-algebra of all bounded linear operators on $\clh$. For $T \in \clb(\clh)$, $\cln(T)$ and $\clr(T)$ stand for the kernel and range of $T$, respectively. An operator $T \in \clb(\clh)$ is said to be a contraction if $\|Th\|\leq \|h\| $ for all $h \in \clh$, and it is  completely non-unitary (c.n.u.\ in short) if there does not exist any nonzero reducing subspace $\cll$ of $\clh$ such that $T|_{\cll}$ is unitary. A contraction $T$ is pure if $\|T^{*m}h\| \rightarrow 0$ for all $h \in \clh$ as $m \rightarrow \infty$. We say that $T \in\mathcal{B(H)}$ is an isometry if  $\Vert Th \Vert=\Vert h\Vert $ for all $h \in \mathcal{H}$, and $T$ is a partial isometry if $\Vert Th\Vert=\Vert h \Vert$ for all $h \in [\cln (T)]^\perp$. An operator $T \in \clb(\clh)$ is an orthogonal projection if $T=T^2=T^*$. A subspace $\clm$ is invariant under $T$ if $T\clm \subseteq \clm$ and $\clm$ is reducing under $T$ if it is invariant under $T$ and $T^*$ both.

Let $\cle$ and $\cle_*$ be two Hilbert spaces. We will denote by $L^2_\cle$  the space of $\cle$-valued square integrable functions on the unit circle $\T$ with respect to the normalized Lebesgue measure. Let $H^2_{\cle}(\D)$ denote the $\cle$-valued Hardy space on the open unit disc $\D$ defined as 
\[
H^2_{\cle}(\D) = \left\{ f=\sum\limits_{m=0}^\infty a_mz^m:\ a_m \in \cle,\  \sum\limits_{m=0}^\infty \|a_m\|^2 <\infty \right\},
\]
and we often identify $H^2_{\cle}(\D)$ (in the sense of radial limits) as a closed subspace of $L^2_\cle$ without making explicit distinction. With this identification $L^2_\cle = H^2_{\cle}(\D) \oplus [H^2_{\cle}(\D)]^{\perp}$. 
Also, $L^\infty_{\clb(\cle,\cle_*)}$  denotes  the algebra of $\clb(\cle,\cle_*)$-valued  bounded functions on $\T$ with respect to operator supremum norm and $H^\infty_{\clb(\cle,\cle_*)}$ denotes the algebra of $\clb(\cle,\cle_*)$-valued bounded analytic functions on $\D$. For $\Phi \in L^\infty_{\clb(\cle,\cle_*)}$, let $L_\Phi$ denote the  Laurent operator  from $L^2_\cle$ to $L^2_{\cle_*}$defined as
\[
L_\Phi h= \Phi h \quad (h \in L^2_\cle).
\]
Let $P_+^{\cle}$ be the orthogonal projection of $L^2_\cle$ onto $H^2_\cle(\D)$. The Toeplitz operator $T_\Phi$ from $H^2_\cle(\D)$ to $H^2_{\cle_*}(\D)$ is defined by 
\[
T_\Phi h=P_+^{\cle_*}(\Phi h) \quad (h \in H^2_\cle(\D)).
\]
In particular, if $\cle=\cle_*$ and $\Phi(z) =zI$, then we use $T_\Phi$ as $M_z^\cle$. We will frequently use $M_z$ if $\cle$ is clear from the context. The Toeplitz  operator $T_\Phi$ is characterized by the operator equation $(M_z^{\cle_*})^*T_\Phi M_z^\cle=T_\Phi$. If $\Phi \in H^\infty_{\clb(\cle,\cle_*)}$, then $T_\Phi$ is called an analytic Toeplitz operator and is characterized by the equation $M_z^{\cle_*} T_\Phi=T_\Phi M_z^\cle$.

Recall some basic definitions which will be used throughout this note.
An operator $T \in \mathcal{B(H)}$ is said to be a power partial isometry if $T^n$ is a partial isometry for all $n \geq 1$. It is a large class of operators including isometries, co-isometries, orthogonal projections and truncated shifts etc. For a  power partial isometry $T$, we write $E_k=T^{*k}T^k$ and $F_k=T^kT^{*k}$ as the initial and final projections for all $k \geq 0$. Recall that  $E_k \geq E_{k+1}$ and $F_k  \geq F_{k+1}$ for all $k \geq 0$. 

\begin{lem}[cf. \cite{halmospowers}] \label{powerpartial-lem1}
Let $T \in \mathcal{B(H)}$ be a power partial isometry. Then
\begin{enumerate}
\item $E_kE_l=E_lE_k$ and  $F_kF_l=F_lF_k$  for all $k,l \geq 0$.
\item $E_kF_l=F_lE_k$ for all $k,l \geq 0$.
\item $TE_{k+1}=E_{k}T$ for all $k \geq 0$.
\item $TF_{k}=F_{k+1}T$ for all $k \geq 0$.
\end{enumerate}
\end{lem}

\NI
Let $k \geq 1$ be any natural number. A truncated shift of index $k$, denoted by $J_k$, is defined on $\mathcal{H}=\underbrace{\clh_0 \oplus\clh_0 \oplus \cdots \oplus\clh_0}_{k\text{-}times}$ as
\[
J_k(x_1,x_2, \ldots, x_k)=(0, x_1, \ldots, x_{k-1}) \quad (x_i \in \clh_0, i =1, \ldots, k).
\]
Here $\clh_0$ is a Hilbert space and $\underbrace{\clh_0 \oplus\clh_0 \oplus \cdots \oplus\clh_0}_{k\text{-} times}$ is identified with $\clh_0 \otimes \C^k$. Note that $J_1=0$.

Let us recall the Halmos and Wallen decomposition theorem for power partial isometry given in \cite{halmospowers}.

\begin{thm}
Let $T \in \clb(\clh)$  be a power partial isometry. Then there exist subspaces $\clh_u, \clh_s, \clh_b$, and $\clh_k\,\, (k \geq 1)$ reducing $T$   and 
\[
\clh= \clh_u \oplus\clh_s \oplus \clh_b \oplus \left(\bigoplus\limits _{k=1}	^\infty \clh_k\right),
\]
such that $T|_{\clh_u}$ is a unitary, $T|_{\clh_s}$ is a unilateral shift, $T|_{\clh_b}$ is a backward shift and $T|_{\clh_k}$ is a truncated shift of index $k$. 
\end{thm}
It is easy to observe from the Halmos and Wallen decomposition that for $k \geq 1$ (see \cite{BA}),
\[
\clh_k= \bigoplus\limits_{n=1}^k\left(E_{k-n}-E_{k-n+1}\right)\left(F_{n-1}-F_n\right)\clh.
\]

\begin{defn}[{\it Contractive analytic function}]
An operator-valued analytic function  $\Theta : \D \rightarrow \clb(\E,\E_*)$ is said to be contractive if 
$$\Vert \Theta( z)a \Vert\leq\Vert a \Vert \quad (a \in \E),$$
and purely contractive if it also follows $\Vert  \Theta(0) a \Vert < \Vert a \Vert\,\,  \,\, (a \in \cle,\,\, a \neq 0)$.
\end{defn}

\begin{defn}[{\it Inner function}]
A contractive analytic function $\Theta : \D \rightarrow \clb(\E,\E_*)$ is called inner if $\Theta(e^{it})$ is an isometry from $\cle$ to $\cle_*$ almost everywhere (a.e.\ in short) on $\T$.
\end{defn}

\begin{defn}[{\it Characteristic function}]
For a contraction $T$ on $\clh$, define the defect operators $D_T=(I-T^*T)^{{1}\over{2}}$ and $ D_{T^*}=(I-TT^*)^{{1}\over{2}}$  with defect spaces $\cld_T=\overline{\clr(D_T)}$ and $\cld_{T^*}=\overline{\clr(D_{T^*})}$. Then the characteristic function of $T$ is the purely contractive analytic function $\Theta_T: \D \rightarrow \clb(\cld_T,\cld_{T^*})$ defined by 
\[
\Theta_T( z)=\left[-T+ z D_{T^*}(I- z T^*)^{-1}D_T\right]\big|_{\cld_T}
\]
for $ z \in \D$.
\end{defn}

Let $\Theta: \D \rightarrow \clb(\E,\E_*)$ and $\Phi : \D \rightarrow \clb(\F,\F_*)$ be two contractive analytic functions. They are said to coincide if there exist unitary operators $\tau: \E \rightarrow \F$ and $\tau_*: \E_* \rightarrow \F_*$ such that $\tau_*\Theta( z)=\Phi( z)\tau$ for all $z \in \D$. It is well known that two c.n.u.\ contractions $T$ on $\clh$ and $S$ on $\clk$ are unitarily equivalent if and only if their characteristic functions coincide (see \cite{sznagy}).

\section{Characteristic function}{\label{sec3}}

In this section, we shall discuss the characteristic function of a power partial isometry. More specifically, we obtain the characteristic function of a power partial isometry and observe that each coefficient in the characteristic function is a partial isometry. Conversely, a purely contractive analytic function with partially isometric coefficients generates a power partial isometry. As an application, we get the diagonal matricial representation of a class of operators whose characteristic functions coincide with the contractive analytic polynomial with partially isometric coefficients.

Recall the following result given in \cite{sznagy} (see Chapter VI), proof of which is straightforward.

\begin{lem}\label{char-fac-lem1}
Let $T_n$ be a contraction on a Hilbert space $\clh_n$ for $n \geq1$.
Let $\clh=\bigoplus\limits_{n=1}^\infty \clh_n$ and $T =\bigoplus \limits_{n=1}^\infty T_n \in \clb(\clh)$. Then
\[
D_T=\bigoplus\limits_{n=1}^\infty D_{T_n},\quad D_{T^*}=\bigoplus\limits_{n=1}^\infty D_{T_n^*},\quad \cld_T=\bigoplus\limits_{n=1}^\infty \cld_{T_n},\quad\cld_{T^*}=\bigoplus\limits_{n=1}^\infty \cld_{T_n^*},
\]
and hence
\[
\Theta_{T}(z)=\bigoplus\limits_{n=1}^\infty \Theta_{T_n}(z) \quad\quad (z \in \D).
\]
\end{lem}

Let $T$ be a power partial isometry on $\clh$. Following Halmos-Wallen decomposition, we have $T$-reducing subspaces $\clh_u$, $\clh_b$, $\clh_s$ and 
$\clh_k$ $(k \geq 1)$ such that
\[
\clh=\clh_u\oplus\clh_s\oplus\clh_b \oplus \left(\bigoplus\limits_{k=1}^\infty \clh_k\right).
\]
Also, $T_u=T|_{\clh_u}$ is a unitary, $T_s=T|_{\clh_s}$ is a unilateral shift, $T_b=T|_{\clh_b}$ is a backward shift and $T_k=T|_{\clh_k}$ is a truncated shift of index $k$. Now using the above Lemma \ref{char-fac-lem1}, we readily get
\[
\dt_T=\dt_{ T_b}\oplus \left(\bigoplus\limits_{k=1}^\infty \dt_{ T_k}\right).
\]
Since $D_{T_b^{*}} = 0$, $\Theta_{T_b}( z)=0 \,\, \forall \,\,  z \in \D $ and hence the characteristic function of $T$ 
\[
\Theta_T( z)=\bigoplus\limits_{k= 1}^\infty \Theta_{T_k}( z) \qquad (z \in \D).
\]
Observe that for $k\geq 1,$
\[
\cld_{T_k}=(I-T_k^*T_k)\clh_k= (E_0-E_1)\clh_k, 
\]
where 
\[
\clh_k=\bigoplus\limits_{n=1}^k \left(E_{k-n}-E_{k-n+1}\right)\left(F_{n-1}-F_n\right)\clh.
\]
Note that for $n> k$
\[
(E_0-E_1)\left(E_{k-n}-E_{k-n+1}\right) =0.
\]
Thus
\[
\cld_{T_k}= \bigoplus\limits_{n=1}^k (E_0-E_1) \left(E_{k-n}-E_{k-n+1}\right)\left(F_{n-1}-F_n\right)\clh = (E_0-E_1)(F_{k-1}-F_k)\clh.  
\]
Hence  
\[
\dt_{T_k}=\clr ( (E_0-E_1)(F_{k-1}-F_k)).
\]
Similarly, we can prove 
\[
\dt_{T_k^*}=\clr  ((E_{k-1}-E_k)(F_0-F_1)).
\]
Recall that $\cln(T_k) = \cld_{T_k}$ and hence the characteristic function of $T_k$
is
\begin{align*}
\Theta_{T_k}( z) & = \left( -T_k+ z(I-T_kT_k^*)(I- z T_k^*)^{-1}(I-T_k^*T_k )\right)|_{\dt_{T_k}} \\
& = z(I-TT^*)\left(I+ z T^*+\cdots+ z^{k-1} T^{*(k-1)}\right)(I-T^*T)|_{\dt_{T_k}}\\
& = (F_0 -F_1)\left(zI+ z^2 T^*+ \cdots + z^{k} T^{*(k-1)}\right)(E_0-E_1)|_{\dt_{T_k}}\\
\end{align*}
for all $z \in \D$. Now using Lemma \ref{powerpartial-lem1}, for any $l \geq 0$,
we have
\begin{align*}
(F_0 -F_1)T^{*l}(E_0-E_1)(F_{k-1}-F_k) &= T^{*l}(F_l -F_{l+1})(E_0-E_1)(F_{k-1}-F_k)\\
&= T^{*l}(E_0-E_1)(F_l -F_{l+1})(F_{k-1}-F_k)\\
&=\begin{cases}
T^{*l}(E_0-E_1)(F_{l}-F_{l+1}), &  \mbox{if~} l=k-1\\
0 &  \mbox{if~} l \neq k-1.\\
\end{cases}
\end{align*}
Therefore, for $ g \in \clh$
\[
\Theta_{T_k}( z)((E_0-E_1)(F_{k-1}-F_k)g)= T^{*k-1}(E_{0}-E_1)(F_{k-1}-F_k)g z^k \quad (z \in \D). 
\]
Now for $f=f_0 +\sum\limits_{k=1}^\infty f_k \in \dt_T$, where $f_0 \in \dt_{T_b}$ and $f_k \in \dt_{T_k}$, the characteristic function of $T$ becomes
\[
\Theta_T( z)f=\sum\limits _{k=1}^\infty   T^{*(k-1)}f_kz^k=\sum\limits _{k=1}^\infty P_{\mathcal{H}_k} T^{*(k-1)}fz^k,
\]
where $P_{\mathcal{H}_k}$ is the orthogonal projection of $\clh$ onto $\mathcal{H}_k$. Now set $C_k = P_{\mathcal{H}_k} T^{*(k-1)}$ for $k \geq 1$. Since $T$ is a power partial isometry and $\clh_k$ reduces $T$ for each $k$, we get
\[
C_kC_k^*C_k = P_{\mathcal{H}_k} T^{*(k-1)} T^{(k-1)}T^{*(k-1)} = P_{\mathcal{H}_k} T^{*(k-1)}= C_k.
\]
Therefore, each $C_k$ is a partial isometry. Furthermore, note that the characteristic function of truncated shift of index $k$ is a monomial of degree $k$ whose coefficient is also partial isometry. \\

We record the aforementioned discussion in the following.

\begin{thm} \label{mainthm}
Let $T$ be a power partial isometry on $\clh$. Then the characteristic function of $T$ is $\Theta_T( z) = \left[\sum\limits _{k=1}^\infty P_{\mathcal{H}_k} T^{*k-1}z^k \right]\Big|_{\dt_T}$ for $z \in \D$, where $P_{\mathcal{H}_k}$ is the orthogonal projection of $\clh$ onto $\mathcal{H}_k$. Moreover, each coefficient in the characteristic function is a partial isometry.
\end{thm}

The above result raises the natural question in the following:
\textsf{\textsf{Is the converse of the above result true?}} 
To answer this question, we shall use the following easy yet powerful result.

\begin{lem}\label{Main-Lemma}
Let $\Theta : \D \rightarrow \clb(\E,\E_*)$ be a contractive analytic function such that $\Theta( z)=\sum\limits_{m=0}^\infty  \theta_mz^m,$ where each $\theta_m \in \clb(\cle,\cle_*)$ is a partial isometry for $m \geq 0$. Then $\theta_i^*\theta_j=0_{\cle}$ and $\theta_j\theta_i^*=0_{\cle_*}$ for all $i \neq j$.
\end{lem}

\begin{proof}
 Since $\Theta$ is a contractive analytic function, then for $a \in \cle$,  
\begin{equation} \label{eq1}
\sum\limits_{m=0}^\infty \Vert \theta_m a\Vert ^2 \leq \Vert a\Vert ^2.
\end{equation}
For $a \in \clr ( \theta_i^*)=[\cln  (\theta_i)]^\perp$, $\Vert \theta_i a \Vert=\Vert a\Vert$ as each $\theta_i$ is a partial isometry. By (\ref{eq1}), we get 
\[
\theta_j a=0 \quad \text{for all~~} j \neq i. 
\]
Hence, $\clr  (\theta_i^*) \subseteq \cln  (\theta_j)$ for $j \neq i$. Equivalently, $\theta_j \theta_i^*=0_{\cle_*}$ for all $j \neq i$. 

For the second one, observe that if $\Theta$ is a contractive analytic function, then $\tilde{\Theta}( z):=\sum\limits_{m=0}^\infty   \theta_m^*z^m$ is also a contractive analytic function. By the same argument, $\theta_i^*\theta_j=0_\cle$ for all $i \neq j$.
\end{proof}

\begin{lem}\label{Function-Lemma}
Let $\Theta : \D \rightarrow \clb(\E,\E_*)$ be a purely contractive analytic function such that 
\[
\Theta( z)=\sum_{m=1}^\infty  \theta_mz^m \quad (z \in \D),
\]
where $\theta_m \in \mathcal{B}(\E,\E_*)$ are partial isometries for $m\geq 1$. Then for any
$\displaystyle \sum_{n=0}^{\infty}a_n z^n \in H^2_{\E_*}(\D)$,
\[
\left(I - T_{\Theta}T_{\Theta}^*\right)\left(\displaystyle \sum_{n=0}^{\infty}a_n z^n \right) 
= a_0 +  \sum_{n=1}^{\infty} \left( I - \sum_{m=1}^{n}\theta_m \theta_m^* \right) a_nz^n.
\]
\end{lem}

\begin{proof}
First note that for any $a \in \cle_*$ and $k \geq 1$, 
\[
	T_{\Theta}^*(az^k)=\sum\limits_{m=1}^k \theta_m^*az^{k-m}=\sum\limits_{m=0}^{k-1} \theta_{k-m}^*az^m.
\]
Then

\begin{align*}
T_{\Theta}T_{\Theta}^* \left(a z^k \right) 
& = T_{\Theta} \left(\sum\limits_{m=0}^{k-1} \theta_{k-m}^*az^m\right)\\
&=\sum\limits_{m=0}^{k-1}T_\Theta(\theta_{k-m}^*az^m)\\
&=\sum\limits_{m=0}^{k-1}\left(\sum\limits_{l=m+1}^\infty \theta_{l-m}\theta_{k-m}^*az^l\right)\\
&=\left(\sum\limits_{m=0}^{k-1}\theta_{k-m}\theta_{k-m}^*a\right)z^k\\
&=\left(\sum\limits_{m=1}^{k}\theta_m\theta_m^*a\right)z^k,
\end{align*}
where the second last equality follows using Lemma \ref{Main-Lemma}. Also note that $T_\Theta^*(a)=0$ for all $a \in \cle_*$.
Therefore, for $\sum\limits_{n=0}^\infty a_n z^n \in H^2_{\cle_*}(\D)$, 
\begin{align*}
\left(I - T_{\Theta}T_{\Theta}^* \right)\left(\displaystyle \sum_{n=0}^{\infty}a_n z^n \right) 
&= \displaystyle \sum_{n=0}^{\infty}a_n z^n - \sum_{n=1}^{\infty} \left(\sum_{m=1}^{n} \theta_m \theta_m^* a_n\right) z^{n}\\
& = a_0 +  \sum_{n=1}^{\infty} \left( I - \sum_{m=1}^{n}\theta_m \theta_m^* \right) a_nz^n.
\end{align*}
\end{proof}

Returning to the above question, first suppose that $\Theta : \D \rightarrow \clb(\E,\E_*)$ is a contractive analytic function such that 
\[
\Theta( z)=\sum_{m=0}^\infty  \theta_mz^m\quad (z \in \D),
\]
$\theta_m \in \mathcal{B}(\E,\E_*)$ are partial isometries for all $m \geq  0$ and $\theta_0 \neq 0$. Then 
\[
\|\Theta(0)a\| = \|\theta_0 a\| = \| a\|\,\, \quad (\forall \,\, a \in [\cln(\theta_0)]^\perp).
\]
Then $\Theta$ can not be purely contractive.
Thus, if $\Theta$ is purely contractive with partially isometric Fourier coefficients, then $\Theta(0)=0$.

We are now ready to state the main result of this section. Our proof is inspired by Sz.-Nagy-Foia\c{s}' model theory (see \cite[Chapter VI]{sznagy}).

\begin{thm}\label{mainthm}
Let $\Theta : \D \rightarrow \clb(\E,\E_*)$ be a purely contractive analytic function such that 
\[
\Theta( z)=\sum_{m=1}^\infty  \theta_mz^m \quad (z \in \D),
\]
where $\theta_m \in \mathcal{B}(\E,\E_*)$ are partial isometries for $m \geq 1$. Then there exist a Hilbert space 
\[
\clh=\left\{(I-T_\Theta T_\Theta^*)f\oplus \Delta_{\Theta} g: f \in H^2_{\E_*}(\D), g \in [\he]^\perp\right\},
\]
where $\Delta_{\Theta}$ is a constant projection, and a c.n.u. power partial isometry $T$ on $\clh$ defined by
\[
T^*(u \oplus v) = e^{-it}(u(e^{it}) - u(0)) \oplus e^{-it} v(e^{it}) \quad (u \oplus v \in \clh)
\]
such that the characteristic function of $T$ coincides with $\Theta$.
\end{thm}

\begin{proof}
Suppose
\[
\Theta( z)=\sum_{m=1}^\infty \theta_mz^m    \quad\quad (z \in \D),
\]
where $\theta_m \in \mathcal{B}(\E,\E_*)$ are partial isometries for $m \geq 1$. 
Define an operator-valued function $\Delta_{\Theta} \in L^\infty_{\clb(\cle)}$ by 
\[
\Delta_{\Theta}(e^{it})=\left[I -\Theta(e^{it})^*\Theta(e^{it})\right]^{1 \over 2} \quad (\text{a.e.\ on } \T).
\]
Now using  Lemma \ref{Main-Lemma}, $\theta_i^* \theta_j=0_{\cle}$ and $\theta_i \theta_j^*=0_{\cle_*}$ for all $i \neq j$. Thus, for $\eta, \zeta \in \cle$ and $n\in \Z\setminus\{0\}$,
\begin{align*}
0=\langle \Theta\eta, z^n \Theta \zeta\rangle&=\frac{1}{2\pi}\int_{0}^{2\pi} \left\langle \Theta(e^{it})\eta,e^{int}\Theta(e^{it})\zeta\right\rangle_{\cle_*} dt\\
& = \frac{1}{2\pi}\int_{0}^{2\pi} e^{-int} \left\langle \Theta(e^{it})^*\Theta(e^{it})\eta,\zeta\right\rangle_{\cle} dt\\
&=\left\langle  \frac{1}{2\pi}\int_{0}^{2\pi} e^{-int} \Theta(e^{it})^*\Theta(e^{it})\eta  \ dt,\zeta\right\rangle_{\cle} \quad \text{(See \cite[Appendix\ A]{raul})}.
\end{align*}
Since it is true for all $\zeta \in \cle$, we have
\[
\frac{1}{2\pi}\int_{0}^{2\pi} e^{-int} \left(\Theta(e^{it})^*\Theta(e^{it})\eta\right)\, dt=0, \quad  (\forall\ \eta \in \cle, n \in \Z\setminus \{0\}).
\]
Thus $\Theta(e^{it})^*\Theta(e^{it})\eta$ is constant a.e.\ on $\T$, call it $\xi_{\eta}$. Then
\[
\langle \xi_\eta, \eta\rangle_{\cle}=\frac{1}{2\pi}\int_{0}^{2\pi}\left\langle \Theta(e^{it})^*\Theta(e^{it})\eta, \eta\right\rangle_{\cle} dt =\frac{1}{2\pi}\int_{0}^{2\pi} \|\Theta(e^{it})\eta\|_{\cle_*}^2 dt = \|\Theta \eta\|_{H^2_{\cle_*}(\D)}^2=\sum\limits_{m=1}^\infty \|\theta_m \eta\|_{\cle_*}^2.
\]
So we get $\Theta(e^{it})^*\Theta(e^{it})\eta=\xi_\eta=\sum\limits_{m=1}^\infty \theta_m^*\theta_m \eta.$ Note that this infinite series $\sum\limits_{m=1}^\infty \theta_m^*\theta_m$ is the strong operator limit of $\sum\limits_{m=1}^n \theta_m^*\theta_m$  as $ n \rightarrow \infty$ and hence a projection (see \cite[Chapter II, Problem 3.5]{conway}). Thus $\Theta(e^{it})^*\Theta(e^{it})$ is a constant projection a.e.\ on $\T$ and  so is $I -\Theta(e^{it})^*\Theta(e^{it})$. 
Therefore, 
\[
\Delta_{\Theta}(e^{it})=\Delta_{\Theta}(e^{it})^2 = \Delta_{\Theta}(e^{it})^{*}.
\]
For the sake of brevity, we write $\Delta_{\Theta}(e^{it})=\Delta_{\Theta}$ as it is a constant projection (independent of $t$). And for the same reason, 
\[
L_{\Delta_\Theta}|_{H^2_\cle(\D)}=T_{\Delta_\Theta}.
\] 
Set
\[
\clk=L^2_{\E_*} \oplus \Delta_{\Theta} L^2_\E, \quad\clk_+=H^2_{\E_*}(\D) \oplus \Delta_{\Theta} L^2_\E,
\]
and 
\[
\mathcal H=\clk_+\ominus \{\Theta h \oplus\Delta_{\Theta} h: h \in \he\}.
\]
Let $U$ denote the multiplication by $e^{it}$ on $\clk$. Then $U$ is unitary. 
Consider $U_+=U|_{\clk_+}$ and let $T$ be an operator on $\mathcal H$ defined by
\[
T^*=U_+^*|_\mathcal{H}.
\]
Thus
\[
T=P_\mathcal{H}U_+|_\mathcal{H},
\]
where $P_\mathcal H$ is the orthogonal projection of $\clk_+$ onto $\clh$. From 
Sz.-Nagy-Foia\c{s} construction, it is clear that $T$ is a c.n.u.\ contraction.

We first describe the model space $\mathcal{H}$ and the model operator $T$. 
Suppose $f \oplus \Delta_{\Theta} g \in \mathcal{H}$, where $f \in H^2_{\cle_*}(\D)$ and $g \in L^2_\cle$. Then for each $h \in \he$,
\begin{align*}
0= \langle f \oplus\Delta_{\Theta} g, \Theta h\oplus \Delta_{\Theta} h \rangle_\clh =  \langle T_\Theta^* f + P_+ \Delta_{\Theta} g,h \rangle_{H^2_{\cle}(\D)}.
\end{align*}
Therefore, $T_\Theta^* f + P_+ \Delta_{\Theta} g=0$, i.e., $P_+\Delta_{\Theta} g=-T_\Theta^*f$. Since $\Delta_{\Theta}$ is constant, 
\[
L_{\Delta_\Theta}P_+g =P_+ L_{\Delta_\Theta} g= P_+\Delta_{\Theta} g=-T_\Theta^*f. 
\]
It follows that $T_\Theta ^*f \in \clr(L_{\Delta_\Theta})$. Therefore,
\[
T_\Theta^*f =L_{\Delta_\Theta} T_\Theta^*f= T_{\Delta_{\Theta}}T_\Theta^*f= 0,
\]
where the last equality holds by the definition of  $\Delta_{\Theta}$. Again, $L_{\Delta_\Theta} P_+g=P_+\Delta_{\Theta} g=0$. Thus
\[
\clh=\left\{(I-T_\Theta T_\Theta^*)f\oplus\Delta_{\Theta} g: f \in H^2_{\E_*}(\D), g \in [\he]^\perp\right\}.
\]
Let $f=\sum\limits_{n=0}^\infty  a_nz^n \in H^2_{\E_*}(\D)$, and $g=\sum\limits_{n=1}^{\infty}  b_n\bar{z}^n\in [\he]^\perp$.
Now by Lemma \ref{Function-Lemma},
\[
(I-T_\Theta T_\Theta^*)f=a_0+\sum\limits_{n=1}^\infty \left(I-\sum\limits_{m=1}^n \theta_m\theta_m^*\right)a_nz^n.
\]
Let 
\[
h=(I -T_\Theta T_\Theta^*)f \oplus \Delta_{\Theta} g \in \clh. 
\]
Then
\begin{align*}
Th &=P_\mathcal{H}U_+\left(a_0+\sum\limits_{n=1}^\infty\left(I-\sum\limits_{m=1}^n \theta_m\theta_m^*\right)a_ne^{int}\oplus\sum\limits_{n=1}^{\infty} \Delta_{\Theta}b_ne^{-int}\right)\\
 &=P_\mathcal{H}\left(a_0e^{it}+\sum\limits_{n=1}^\infty \left(I-\sum\limits_{m=1}^n \theta_m\theta_m^*\right)a_ne^{i(n+1)t}\oplus\sum\limits_{n=1}^{\infty} \Delta_{\Theta}b_ne^{i(-n+1)t}\right)\\
&=\sum\limits_{n=0}^\infty \left(I- \sum\limits_{m=1}^{n+1}\theta_m\theta_m^*\right)a_ne^{i(n+1)t}\oplus\sum\limits_{n=2}^{\infty}\Delta_{\Theta}b_ne^{i(-n+1)t}.
\end{align*}
Also, for $p \geq 1$, we obtain
\begin{align*}
T^{*p}T^pT^{*p}h &=T^{*p}T^p\left(\sum\limits_{n=p}^\infty \left(I-\sum\limits_{m=1}^n \theta_m\theta_m^*\right)a_ne^{i(n-p)t}\oplus\sum\limits_{n=1}^{\infty}\Delta_{\Theta}b_ne^{i(-n-p)t}\right)\\
&=T^{*p}\left(\sum\limits_{n=p}^\infty\left(I-\sum\limits_{m=1}^n \theta_m\theta_m^*\right)a_ne^{int}\oplus\sum\limits_{n=1}^\infty\Delta_{\Theta}b_ne^{-int}\right)\\
&= \sum\limits_{n=p}^\infty \left(I-\sum\limits_{m=1}^n \theta_m\theta_m^*\right)a_ne^{i(n-p)t}\oplus\sum\limits_{n=1}^{\infty}\Delta_{\Theta}b_ne^{i(-n-p)t}\\
&=T^{*p}((I-T_\Theta T_\Theta^*)f\oplus\Delta_{\Theta} g)=T^{*p}h.
\end{align*}
Therefore, $T$ is a power partial isometry on $\clh$.

Our remaining task is to show that the characteristic function of $T$ coincides with $\Theta$. In order to do that, first we have to find defect spaces, namely, 
\[
\mathcal{D}_T=\clr (I-T^*T)=\cln(T) \quad \mbox{and} \quad \mathcal{D}_{T^*}=\clr (I-TT^*)=\cln (T^*).
\]
Now $h=(I-T_\Theta T_\Theta^*)f\oplus\Delta_{\Theta} g\in \cln(T)$ if and only if
\[
\sum\limits_{n=0}^\infty \left(I- \sum\limits_{m=1}^{n+1}\theta_m\theta_m^*\right)a_ne^{i(n+1)t}\oplus\sum\limits_{n=2}^{\infty}\Delta_{\Theta}  b_ne^{i(-n+1)t}=0.
\]
Equivalently,
\begin{align*}
\left(I-\sum\limits_{m=1}^{n+1}\theta_m\theta_m^*\right)a_n & =0\,\, ~~ \forall \,n \geq 0,\\
\mbox{and} \quad 
\Delta_{\Theta} b_n & = 0\,\, ~~ \forall \,n \geq 2.
\end{align*}
The former equality says that 
\begin{align*}
a_0 & = \theta_1\theta_1^*a_0 \\
\mbox{and} \quad \left(I-\sum\limits_{m=1}^n \theta_m\theta_m^*\right)a_n & = \theta_{n+1}\theta_{n+1}^*a_n \,\, ~~  \forall \, n \geq 1.
\end{align*}
Therefore,
\[
\mathcal{D}_T=\left\{\sum\limits_{n=0}^\infty \theta_{n+1}\theta_{n+1}^*a_nz^n \oplus\Delta_{\Theta} b\bar{z}: (a_n) \in \ell^2(\E_*), b \in \E\right\}.
\]
Now $h \in \cln(T^*)$ if and only if 
\[
\sum\limits_{n=1}^\infty\left(I-\sum\limits_{m=1}^n \theta_m\theta_m^*\right)a_ne^{i(n-1)t}\oplus\sum\limits_{n=1}^{\infty}\Delta_{\Theta}b_n e^{i(-n-1)t}=0,
\] 
i.e.,
\begin{align*}
\left(I-\sum\limits_{m=1}^n \theta_m\theta_m^*\right)a_n & =0\,\, ~~ \forall \,n \geq 1 \\
\mbox{and} \quad 
\Delta_{\Theta}b_n & =0\,\, ~~ \forall \,n \geq 1.
\end{align*}
Hence, $\mathcal{D}_{T^*}=\cle_*$.

We can now proceed to determine the characteristic function $\Theta_T:\D \rightarrow \clb(\mathcal{D}_T, \mathcal{D}_{T^*})$ of $T$.
Let 
\[
h=\sum\limits_{n=0}^\infty \theta_{n+1}\theta_{n+1}^*a_nz^n +\Delta_{\Theta}b\bar{z} \in \mathcal{D}_T.
\]
Then, for $z \in \D,$
\begin{align*}
\Theta_T( z)(h)&=\left(-T+ z(I-TT^*)(I- z T^*)^{-1}\right)\left(\sum\limits_{n=0}^\infty \theta_{n+1}\theta_{n+1}^*a_ne^{int} \oplus\Delta_{\Theta}be^{-it}\right)\\
&= z P_{\E_*}\left(\sum \limits_{j=0}^\infty  T^{*j}z^j\right)\left(\sum\limits_{n=0}^\infty \theta_{n+1}\theta_{n+1}^*a_ne^{int} \oplus\Delta_{\Theta}be^{-it}\right)\\
&=  z P_{\E_*}\left(\sum\limits_{j=0}^\infty  \left(\sum\limits_{n=j}^\infty \theta_{n+1}\theta_{n+1}^*a_ne^{i(n-j)t} \right) z^j   \oplus\sum\limits_{j=0}^\infty\Delta_{\Theta}be^{-i(j+1)t}z^j\right)\\
&= z\left(\sum\limits_{j=0}^\infty \theta_{j+1}\theta_{j+1}^*a_jz^j\right)\\
&= \sum\limits_{j=0}^\infty  \theta_{j+1}\theta_{j+1}^*a_jz^{j+1}.
\end{align*}

Now we will show that $\Theta_T$ coincides with $\Theta$. Define a map $\tau: \E \rightarrow \mathcal{D}_T$  by 
\[
\tau(a)=\sum\limits_{n=0}^\infty \theta_{n+1}az^n\oplus\Delta_{\Theta}a\bar{z} \quad (a \in \cle).
\]
Then $\tau$ is a well-defined linear map. Also,
\begin{align*}
\|\tau(a)\|^2&=\sum\limits_{n=0}^\infty\|\theta_{n+1}a\|^2+\left\| \Delta_{\Theta}a\right\|^2\\
&=\sum\limits_{n=0}^\infty\|\theta_{n+1}a\|^2+ \|a\|^2- \sum\limits_{m=1}^\infty \|\theta_ma\|^2\\
&=\|a\|^2.
\end{align*}
Thus $\tau$ is an isometry. To prove $\tau$ is surjective as well,
let 
\[
h=\sum\limits_{n=0}^\infty \theta_{n+1}\theta^*_{n+1}a_nz^n \oplus\Delta_{\Theta}b\bar{z}  \in \cld_T,
\] 
where $(a_n) \in \ell^2(\E_*)$ and $b \in \cle.$ For $a= \sum\limits_{m=0}^\infty\theta^*_{m+1}a_m+\Delta_{\Theta}b \in \cle$, it is easy to check that 
$\tau(a)=h$. Hence $\tau$ is a unitary.

Also, let $\tau_*: \cld_{T^*} \rightarrow \cle_*$ be the identity map, that is, 
$\tau_*(a) = a$. Now, for $ a \in \cle$ and $z \in \D$,
\begin{align*}
\tau_*\Theta_T( z)\tau(a)&= \tau_*\Theta_T( z)\left(\sum\limits_{n=0}^\infty \theta_{n+1}az^n \oplus\Delta_{\Theta}a\bar{z}\right)\\
&=\tau_*\left(\sum\limits_{n=0}^\infty \theta_{n+1}a z^{n+1}\right)\\
&=\sum\limits_{n=1}^\infty \theta_{n}a z^{n}\\
&=\Theta( z)a.
\end{align*}
This finishes the proof.
\end{proof}

It is well known that the characteristic function is inner if and only if the corresponding contraction is pure (see \cite{sznagy}). The following corollary is in that direction, and we omit the proof because it follows a similar approach to the proof of the above Theorem \ref{mainthm}.

\begin{cor}
Let $\Theta: \D \rightarrow \clb(\cle,\cle_*)$ be a purely contractive analytic inner function such that 
\[
\Theta( z)=\sum\limits_{m=1}^\infty  \theta_mz^m \quad (z\in \D),
\]
where each $\theta_m \in \clb(\cle,\cle_*)$ is a partial isometry for $m \geq 1$. Then there exist a Hilbert space  
\[
\clh=\left\{(I-T_\Theta T_\Theta^*)f: f \in H^2_{\E_*}(\D) \right\}
\]
and a pure power partial isometry $T$ on $\clh$ defined by
\[
T^*u = e^{-it}\left(u(e^{it}) - u(0)\right)  \quad (u \in \clh)
\]
such that the characteristic function of $T$ coincides with $\Theta$.
\end{cor}

As an application of Theorem \ref{mainthm}, let us see the case of contractive analytic polynomial with partially isometric coefficients. In \cite{foias}, Foia\c{s} and Sarkar characterized the c.n.u.\ contractions with polynomial characteristic functions and proved that such operators have an upper triangular matricial representation of the form 
\[
\begin{pmatrix}
S & * & *\\
0 & N & *\\
0 & 0 & C
\end{pmatrix}
\]
where $S$, $N$, and $C$ are  unilateral shift, nilpotent and backward shift, respectively. In our case, since the corresponding c.n.u.\ contraction is a power partial isometry, we get a block diagonal representation and also we find the Halmos--Wallen decomposition spaces explicitly.

\begin{thm}{\label{spaces}}
In the setting of Theorem \ref{mainthm}, let $\Theta: \D \rightarrow \clb(\cle,\cle_*)$ be a  contractive analytic polynomial of degree $k$ such that 
\[
\Theta( z)=\sum\limits_{m=1}^k \theta_m z^m \quad (z \in \D),
\]
where each $\theta_m \in \clb(\cle,\cle_*)$ is a partial isometry for $m \geq 1$ and $T$ on the Hilbert space $\clh$ is the corresponding c.n.u.\ power  partial isometry. Then there exist $T$-reducing subspaces $\clh_s= \left(I-\sum\limits_{m=1}^k \theta_m \theta_m^*\right) H^2_{\cle_*}(\D)$, 
$\clh_b = \left(I-\sum\limits_{m=1}^k \theta_m^* \theta_m \right)[H^2_\cle(\D)]^\perp$ and $\clh_t=\clh \ominus (\clh_s \oplus \clh_b)$ such that $\clh=\clh_s \oplus \clh_t \oplus \clh_b$ and 
\[
T=
\begin{pmatrix}
S &0 &0\\
 0&N &0\\
0& 0& C
\end{pmatrix},
\]
where $S \in \clb(\clh_s)$ is a unilateral shift, $N \in \clb(\clh_t)$ is nilpotent of index $k$, and  $C \in \clb(\clh_b)$ is a backward shift.
\end{thm}

\begin{proof}
From Theorem \ref{mainthm}, we get
\[
\clh=\{(I-T_\Theta T_\Theta^*)f \oplus\Delta_{\Theta} g: f \in H^2_{\cle_*}(\D), g \in [H^2_\cle(\D)]^\perp\}.
\]
Lemma \ref{Main-Lemma} infers that $\theta_j^*\theta_i=0_{\cle}$ and $\theta_i \theta_j^*=0_{\cle_*}$ for $i \neq j$. Moreover, 
\[
\Delta_{\Theta} = I-\sum\limits_{m=1}^k \theta_m^* \theta_m \quad \mbox{and} \quad \Delta_{\Theta^*}= I-\sum\limits_{m=1}^k \theta_m \theta_m^*
\]
are orthogonal projections. Set
\[
\clh_b=\Delta_{\Theta} [H^2_\cle(\D)]^\perp.
\]
Then it is trivial to check that $\clh_b$ is invariant under $T^*$. Also, as in proof of Theorem \ref{mainthm}, for $g=\sum\limits_{n=1}^{\infty} b_n\bar{z}^n \in [\he]^\perp$, 
\[
T\left(\Delta_{\Theta} g\right)=\Delta_{\Theta}\left(\sum\limits_{n=2}^{\infty} b_ne^{i(-n+1)t}\right) \in \Delta_{\Theta} [H^2_\cle(\D)]^\perp.
\]
Hence $\clh_b$ reduces $T$. 
Now consider 
\begin{align*}
TT^*\left(\Delta_{\Theta} g\right)&=T\left(\sum\limits_{n=1}^{\infty}\Delta_{\Theta}  b_ne^{i(-n-1)t}\right)\\
&=\Delta_{\Theta}\left(\sum\limits_{n=1}^{\infty}b_ne^{-int}\right)= \Delta_{\Theta} g.
\end{align*}
Also note that 
\[
\left\Vert T^p(\Delta_\Theta g)\right\Vert^2=\left\Vert\sum\limits_{n=p+1}^{\infty}\Delta_\Theta b_ne^{i(-n+p)t}\right\Vert^2  \leq  \sum\limits_{n=p+1}^{\infty}\left\Vert b_n\right\Vert^2 \rightarrow 0
\]
as $p \rightarrow \infty$. Hence $T|_{\clh_b}$ is a backward shift.
Define 
\[
\clm=\clh \ominus \clh_b =(I-T_\Theta T_\Theta^*)H^2_{\cle_*}(\D).
\]
Now we define another space $\clh_s$ as
\[
\clh_s=\Delta_{\Theta^*} H^2_{\cle_*}(\D).
\]
Note that $\clh_s$ is a subspace of $\clm$ and it can be proved by using the following fact:
\[
\left(I-\sum\limits_{l=1}^n \theta_l \theta_l^*\right)\left( I-\sum\limits_{m=1}^k \theta_m \theta_m^*\right)=\left( I-\sum\limits_{m=1}^k \theta_m \theta_m^* \right)\,\, \forall \,\, n\leq k.
\]
For $f=\sum\limits_{n=0}^\infty a_nz^n \in H^2_{\cle_*}(\D)$, using Lemma \ref{Function-Lemma},
\begin{align*}
T(\Delta_{\Theta^*} f)&=\sum\limits_{n=0}^{k-1} \left( I-\sum\limits_{m=1}^{n+1} \theta_m \theta_m^* \right)\Delta_{\Theta^*}a_ne^{i(n+1)t}+\sum\limits_{n=k}^\infty\Delta_{\Theta^*} a_ne^{i(n+1)t}\\
&=\Delta_{\Theta^*}\left(\sum\limits_{n=0}^\infty a_n e^{i(n+1)t}\right) \in \Delta_{\Theta^*} H^2_{\cle_*}(\D).
\end{align*}
And $T^*\clh_s \subseteq \clh_s$ is trivial to prove. Hence, $\clh_s$ reduces $T$. Also,
\begin{align*}
T^*T(\Delta_{\Theta^*} f)&=T^*\left(\Delta_{\Theta^*} \left(\sum\limits_{n=0}^\infty a_ne^{i(n+1)t}\right)\right)\\
&= \Delta_{\Theta^*} \left(\sum\limits_{n=0}^\infty a_ne^{int}\right)=\Delta_{\Theta^*} f.
\end{align*}
Therefore, $T|_{\clh_s}$ is an isometry. Furthermore, it is pure as $T^{*p}=U_+^{*p}|_{\clh}$ for all $p \geq 0$.
Finally, set $\clh_t=\clm \ominus \clh_s$. Let $\theta_0=0$ and $\theta_n =0$ for all $n > k$. Then $(I-T_\Theta T_\Theta^*)f  \in \clh_t$ if and only if
\begin{align*}
&\left \langle (I-T_\Theta T_\Theta^*)f, \Delta_{\Theta^*} f' \right \rangle =0\quad \forall \,\, f'\in H^2_{\cle_*}(\D)\\
\Leftrightarrow& \left\langle \sum\limits_{n=0}^\infty \left(I-\sum\limits_{m=0}^n \theta_m\theta_m^* \right)a_nz^n, \Delta_{\Theta^*} f' \right\rangle \quad \forall\,\, f'\in H^2_{\cle_*}(\D)\\
\Leftrightarrow & \left\langle \Delta_{\Theta^*}f,f'\right \rangle = 0\quad \forall \,\,f'\in H^2_{\cle_*}(\D).
\end{align*}
Hence we get $\Delta_{\Theta^*} f =0$, i.e., $\sum\limits_{m=0}^k \theta_m\theta_m^*a_n=a_n$ for all $n \geq 0$.

Therefore, 
\begin{align*}
(I-T_\Theta T_\Theta^*) f&=\sum\limits_{n=0}^{k-1}\left(I-\sum\limits_{m=0}^n \theta_m\theta_m^*\right)\left(\sum\limits_{m=0}^k \theta_m\theta_m^*\right)a_nz^n+\sum\limits_{n=k}^\infty\left(I-\sum\limits_{m=0}^k \theta_m\theta_m^*\right)\left(\sum\limits_{m=0}^k \theta_m\theta_m^*\right)a_nz^n \\
&= \sum\limits_{n=0}^{k-1}\left(\sum\limits_{m=0}^k \theta_m\theta_m^*-\sum\limits_{m=0}^n \theta_m\theta_m^*\right)a_nz^n\\
&=\sum\limits_{n=0}^{k-1} \left(\sum\limits_{m=n+1}^k\theta_m\theta_m^*\right)a_nz^n.
\end{align*}
So, we get 
\[
\clh_t=\left\{\sum\limits_{n=0}^{k-1}\left(\sum\limits_{m=n+1}^k\theta_m\theta_m^*\right)a_nz^n: a_n \in \cle_*, 1\leq n \leq k-1\right\}.
\]
Since $\clh_t=\clh \ominus (\clh_b \oplus \clh_s)$, it is $T$-reducing.
Moreover, for $(I-T_\Theta T_\Theta^*)f \in H_t$,
\begin{align*}
T((I-T_\Theta T_\Theta^*)f)&=P_\clh\left(\sum\limits_{n=0}^{k-1}\left(\sum\limits_{m=n+1}^k \theta_m\theta_m^*\right)a_ne^{i(n+1)t} \right)\\
&= \sum\limits_{n=0}^{k-1}\left(I-\sum\limits_{m=0}^{n+1}\theta_m\theta_m^*\right)\left(\sum\limits_{m=n+1}^k \theta_m\theta_m^*\right)a_ne^{i(n+1)t}\\
&=  \sum\limits_{n=0}^{k-2}\left(\sum\limits_{m=n+2}^k\theta_m\theta_m^*\right)a_ne^{i(n+1)t}.
\end{align*}
Hence $T^k((I-T_\Theta T_\Theta^*)f)=0$, i.e., 
$T|_{\clh_t}$ is a nilpotent operator of index $k$. This completes the proof. 
\end{proof}

The above proof yields something more and there are few remarks in order:

\begin{rem}
For $m\in \{1,2,\ldots,k\}$, set 
\[
\clh_m=\left\{\sum\limits_{n=0}^{m-1}\theta_m\theta_m^* a_nz^n: a_n \in \cle_*, 0 \leq n \leq m-1\right\}.
\]
Then each $\clh_m$ reduces $T$. Indeed, for $g= \sum\limits_{n=0}^{m-1}\theta_m\theta_m^* a_nz^n\in \clh_m$,
\[
Tg=P_\clh\left(\sum\limits_{n=0}^{m-1}\theta_m\theta_m^* a_ne^{i(n+1)t}\right)=\sum\limits_{n=0}^{m-2}\theta_m\theta_m^* a_ne^{i(n+1)t} \in \clh_m.
\]
Also note that $\clh_m=\theta_m\theta_m^*\left(H^2_{\cle_*}(\D)\ominus z^m H^2_{\cle_*}(\D)\right)$. Thus $\clm$ is $T^*$-invariant. It is easy to check that $T|_{\clh_m}$ is a truncated shift of index $m$ for $m \in\{1,2,\ldots,k\}$ and 
\[
\clh_t =\bigoplus\limits_{m=1}^k \clh_m.
\]
\end{rem}

\begin{rem}
If $\Theta: \D \rightarrow \clb(\cle,\cle_*)$ is a contractive analytic inner function such that $\Theta( z)=\sum\limits_{m=0}^\infty \theta_mz^m$ with 
$\theta_i\theta_j^*=0_{\cle_*}$ for all $i \neq j$, then each $\theta_i$ is a partial isometry. Indeed, if $\Theta$ is inner, then $\Theta(e^{it})$ is an isometry from $\E$ to $\E_*$ a.e.\ on $\T$.
Thus, for $a \in \E$, 
\[
a=\Theta(e^{it})^*\Theta(e^{it})a=\sum\limits_{m=0}^\infty \theta_m^*\theta_m a.
\]
Pre-multiplying by $\theta_i$ yields $\theta_i \theta_i^*\theta_i a=\theta_i a\,\, ~~ \forall\,\, i$. Hence each $\theta_i$ is a partial isometry. One can compare this with Lemma \ref{Main-Lemma}.
\end{rem}

\section{Partially isometric Toeplitz operators}{\label{sec4}}

In this section we discuss about the partially isometric Toeplitz operators with operator-valued symbol and a complete characterization of such symbols has been given.

Let $T$ be a power partial isometry on $\clh$. Let $\Theta: \D\rightarrow \clb(\cld_T,\cld_{T^*})$ defined by,
\[
\Theta( z)=\sum\limits_{m=1}^\infty \theta_mz^m
\]
be the characteristic function of $T$. Then each $\theta_m$ is a partial isometry
by Theorem $\ref{mainthm}$. Now consider the Toeplitz operator $T_\Theta$ from $H^2_{\cld_T}(\D)$ to $H^2_{\cld_{T^*}}(\D)$ with operator-valued symbol $\Theta$. 
As noticed above, 
$\Theta(e^{it})^*\Theta(e^{it})$  is a  constant projection a.e.\ on $\T$. Therefore,
\[
T_\Theta T_\Theta^* T_\Theta=T_{\Theta \Theta^*\Theta}=T_\Theta
\]
which implies $T_\Theta$ is a partial isometry. The next natural question one can ask is:
{Characterize $\Gamma \in L^\infty_{\clb(\cld_T,\cld_{T^*})}$ such that, for
$\Phi =\Theta_T + \Gamma$, $T_\Phi$ is  a partial isometry.}

Recently, Sarkar (cf.\ \cite{ss}) characterized the partially isometric Toeplitz operators on $H^2_{\cle}(\D^n)$ for operator-valued symbols and also raised the question of characterizing partially isometric symbols $\Phi\in L^\infty_{\clb(\cle)}$ such that $T_\Phi$ is a partial isometry.
We have given here  some necessary conditions on such symbols  in one variable in Proposition \ref{NC} and a full characterization of such symbols for finite-dimensional underlying spaces. Before proceeding, we have also given a characterization of partially isometric Toeplitz operators $T_\Phi$ from $H^2_\cle(\D)$ to $H^2_{\cle_*}(\D)$ and the proof of this result is similar to that in \cite{ss}, which we had discovered independently without the knowledge of the cited article. For the sake of completion, we present a short proof.

\begin{thm}\label{SSthm}
Let $\Phi\in L^\infty_{\clb(\cle,\cle_*)}$ be nonzero. Then $T_\Phi$ is a  partially isometric Toeplitz operator from $H^2_\cle(\D)$ to $H^2_{\cle_*}(\D)$ if and only if there exist a Hilbert space $\clf$ and inner functions $\Theta(z)=\sum\limits_{m=0}^\infty \theta_mz^m \in H^\infty_{\clb(\clf,\cle_*)}$ and  $\Psi(z)=\sum\limits_{m=0}^\infty \psi_m z^m\in H^\infty_{\clb(\clf,\cle)}$ satisfying $\theta_m\psi_n^*=0$ for all $m,n \geq 1$ such that 
\[
T_\Phi=T_{\Theta}T_{\Psi}^*.
\]
\end{thm}

\begin{proof}
Suppose that $T_\Phi: H^2_\cle(\D) \rightarrow H^2_{\cle_*}(\D)$ is a nonzero partially isometric Toeplitz operator. Then
\[
\clr(T_\Phi^*)=[\cln(T_\Phi)]^\perp=\{f \in H^2_{\cle}(\D):  \|T_\Phi f \|=\|f\|\}.
\]
For $f \in \clr(T_\Phi^*)$,  observe that
\[
	\|M_z^\cle f\|=\|f\|= \|T_\Phi f\|=\left\|(M_z^{\cle_*})^* T_\Phi M_z^\cle f\right\| \leq \|T_\Phi M_z^\cle f\| \leq \|M_z^\cle f\|. 
\]
Therefore, 
\[
\|T_\Phi M_z^\cle f\|=\|M_z^\cle f\|, \quad \text{i.e.}, \quad M_z^\cle f \in [\cln(T_\Phi)]^\perp.
\]
It follows that $[\cln(T_\Phi)]^\perp$ is $M_z$-invariant. Hence the Beurling-Lax-Halmos theorem yields that there exist a Hilbert space $\clf$ and an inner function $\Psi \in H^\infty_{\clb(\clf,\cle)}$ such that
\begin{equation}\label{BLH}
\clr(T_\Phi^*)= T_\Psi H^2_\clf(\D).
\end{equation}
Now, by Douglas' lemma \cite{RGD}, there exist a contraction $X : H^2_{\cle_*}(\D) \rightarrow H^2_{\clf}(\D)$ such that 
\[
T_\Phi^*=T_\Psi X.
\]
This implies that $X =T_\Psi^*T_\Phi^*$ as $T_\Psi$ is an isometry. Set 
$\Theta=\Phi \Psi$. Then $X= T_\Theta^*$. Therefore,
\[
T_\Phi=T_\Theta T_\Psi^*.
\]
Note that
\[
T_\Theta^*T_\Theta=T_\Psi^* T_\Phi^* T_\Phi T_\Psi= T_\Psi^* T_\Psi=I,
\]
where the second last equality holds from (\ref{BLH}) and the fact that $T_\Phi$ is a partial isometry. Thus, $\Theta\in H^\infty_{\clb(\clf, \cle_*)}$ is an inner function. 
Let 
\[
\Theta(z)=\sum\limits_{m=0}^\infty \theta_mz^m  \text{ and } \Psi(z)=\sum\limits_{m=0}^\infty \psi_mz^m
\]
where $\theta_m \in \clb(\clf,\cle_*)$ and $\psi_m \in \clb(\clf,\cle)$ for all $m \geq0$. Since $T_\Phi= T_\Theta T_\Psi^*$ is a Toeplitz operator, we have
\[
(M_z^{\cle_*})^*T_\Theta T_\Psi^* M_z^\cle=T_\Theta T_\Psi^*.
\]
Let $\eta \in \cle$ and $n \geq 0$. Consider
\begin{align*}
(M_z^{\cle_*})^*T_\Theta T_\Psi^* M_z^\cle(\eta z^n)&= (M_z^{\cle_*})^*T_\Theta\left(\sum\limits_{m=0}^{n+1} \psi_m^*\eta z^{n-m+1} \right )\\
&= (M_z^{\cle_*})^*T_\Theta\left(\sum\limits_{m=0}^n \psi_m^*\eta z^{n-m+1}\right) + (M_z^{\cle_*})^*T_\Theta(\psi_{n+1}^*\eta)\\
&=(M_z^{\cle_*})^*M_z^{\cle_*} T_\Theta\left(\sum\limits_{m=0}^n \psi_m^*\eta z^{n-m}\right) + (M_z^{\cle_*})^*T_\Theta(\psi_{n+1}^*\eta)\\
&=T_\Theta T_\Psi^*(\eta z^n) +(M_z^{\cle_*})^*T_\Theta(\psi_{n+1}^*\eta).
\end{align*}
It follows that 
\[
(M_z^{\cle_*})^*T_\Theta(\psi_{n+1} \eta)=\sum\limits_{m=1}^\infty \theta_m \psi_{n+1}^* \eta z^{m-1}=0.
\]
Since $\eta \in \cle$ and $n$ are arbitrary, we obtain $\theta_m \psi_n^*=0$ for all $m,n \geq 1$. 

The converse also holds, which is trivial to prove.
\end{proof}

Let us first see an example for this result:

\begin{ex}
Let $\Theta, \Psi \in H^\infty _{\clb(\C^3)}$ defined by
\[
\Theta(e^{it})=\begin{pmatrix}
e^{it} & 0 &0\\
0 &1 &0\\
0 &0 &e^{it}
\end{pmatrix}
\text{ and }
\Psi(e^{it})=\begin{pmatrix}
0 &0 & 1\\
0 & e^{it} &0\\
1 & 0 &0
\end{pmatrix}.
\]
Then $\Theta$ and $\Psi$ are inner functions and observe that $\theta_1 \psi_1^*=0$. Also,
\[T_\Phi=T_\Theta T_\Psi^*=\begin{pmatrix}
0 & 0 & T_{e^{it}}\\
0 & T_{e^{-it}} &0\\
T_{e^{it}} & 0 &0
\end{pmatrix},
\]
is a partial isometry.
\end{ex}

In order to answer the stated question, first we prove some necessary conditions on $\Phi \in L^\infty_{\clb(\cle,\cle_*)}$ for which $T_\Phi$ is a partial isometry.

\begin{rem}
Define a set 
\[
\bar{H}^\infty_{\clb(\cle,\cle_*)}: = \{\Psi(\bar{z}): \Psi \in H^\infty_{\clb(\cle,\cle_*)}\}.
\]
Then $\bar{H}^\infty_{\clb(\cle,\cle_*)} \subset L^\infty_{\clb(\cle,\cle_*)}$ and hence ${H}^\infty_{\clb(\cle,\cle_*)}+\bar{H}^\infty_{\clb(\cle,\cle_*)} \subset L^\infty_{\clb(\cle,\cle_*)}$. Consider 
\[
\cla : = {H}^\infty_{\clb(\cle,\cle_*)}+\bar{H}^\infty_{\clb(\cle,\cle_*)}.
\]
For $\Phi \in \cla$, in the sequel, we write $\Phi_+(z)=\sum\limits_{m=0}^\infty \phi_mz^m$ and $\Phi_-(z)=\sum\limits_{m=0}^\infty \phi_{-m}\bar{z}^m$ as the analytic and co-analytic part of $\Phi$, where $\phi_n =\widehat{\Phi}(n)\in \clb(\cle,\cle_*)$ for $n \in \Z$.

\end{rem}

\begin{lem}\label{NC1}
Let $\Phi \in L^\infty_{\clb(\cle,\cle_*)}$ be such that $T_\Phi: H^2_\cle(\D) \rightarrow H^2_{\cle_*}(\D)$ is a nonzero partially isometric Toeplitz operator. Then $\Phi(e^{it})$ is a partial isometry a.e.\ on $\T$ and $\Phi \in \cla.$
\end{lem}

\begin{proof}
Suppose $\Phi \in L^\infty_{\clb(\cle,\cle_*)}$ such that $T_\Phi: H^2_\cle(\D) \rightarrow H^2_{\cle_*}(\D)$ is a nonzero partially isometric Toeplitz operator. Then Theorem~ \ref{SSthm} infers that there exist a Hilbert space 
	$\clf$ and inner functions $\Theta(z)=\sum\limits_{m=0}^\infty \theta_mz^m\in H^\infty_{\clb(\clf,\cle_*)}$, $\Psi(z)=\sum\limits_{m=0}^\infty \psi_m z^m\in H^\infty_{\clb(\clf,\cle)}$ satisfying $\theta_m\psi_n^*=0$ for all $m,n \geq 1$ such that 
	\[
	T_\Phi=T_{\Theta}T_{\Psi}^*.
	\]
	Since $T_\Theta T_\Psi^*$ is a Toeplitz operator, $T_\Phi=T_{\Theta {\Psi}^*}$ and hence $\Phi(e^{it})=\Theta(e^{it})\Psi(e^{it})^*$ a.e.\ on $\T$. Now $\Phi(e^{it})$ is a partial isometry a.e.\ on $\T$ as
	\[
	\Phi(e^{it})\Phi(e^{it})^*\Phi(e^{it})= \Theta(e^{it})\Psi(e^{it})^* \Psi(e^{it}) \Theta(e^{it})^*\Theta(e^{it})\Psi(e^{it})^* = \Theta(e^{it})\Psi(e^{it})^* = \Phi(e^{it}).
	\]
	Since $\theta_n \psi_m^*=0$ for all $m,n\geq 1$, 
	\begin{equation*}
		T_\Phi=T_\Theta T_\Psi^*=\Theta(0)T_{\Psi}^*+ T_\Theta \Psi(0)^*-\Theta(0)\Psi(0)^*,
	\end{equation*}
	that is,
	\begin{equation}\label{eq 5}
		\Phi(e^{it})=\theta_0 \Psi(e^{it})^*+\Theta(e^{it}) \psi_0^*-\theta_0\psi_0^* \quad (\text{a.e.\ on }\T).
	\end{equation}
	Thus $\Phi \in \cla$ with $\Phi_+=\Theta\psi_0^*$ and $\Phi_-=\theta_0\Psi^*$.
\end{proof}

Therefore, by virtue of the previous theorem, while characterizing the symbol of a partially isometric Toeplitz operator, it suffices to  restrict our attention to $\Phi \in \cla$ rather than considering $\Phi \in L^\infty_{\clb(\cle,\cle_*)}$ .

\begin{prop}\label{NC}
Let $\Phi \in \cla$ be such that $T_\Phi: H^2_\cle(\D) \rightarrow H^2_{\cle_*}(\D)$ is a nonzero partially isometric Toeplitz operator. Then $\Phi$ satisfies the following conditions:
\begin{enumerate}
\item  $\Phi_+(e^{it})^*\Phi_+(e^{it}) $ and $\Phi_-(e^{it})\Phi_-(e^{it})^*$
 are  operator-valued constant functions a.e.\ on $\T$, where $\Phi_+ \in H^\infty_{\clb(\cle,\cle_*)}$ and $\Phi_-^* \in H^\infty_{\clb(\cle_*,\cle)}$ are contractive analytic functions. \label{con1} 
\item $\hat{\Phi}(m)^*\hat{\Phi}(-n)=0_\cle$ and $\hat{\Phi}(-n)\hat{\Phi}(m)^*=0_{\cle_*}$ for all $m,n \geq 1$.
\label{con2}
\end{enumerate}
\end{prop}

\begin{proof}
Suppose $\Phi \in \cla$ such that $T_\Phi: H^2_\cle(\D) \rightarrow H^2_{\cle_*}(\D)$ is a nonzero partial isometric Toeplitz operator.
Write
\begin{equation}\label{eq 6}
\Phi= \sum\limits_{m=1}^\infty  \phi_{-m}\bar{z}^m +\sum\limits_{m=0}^\infty \phi_mz^{m}.
\end{equation}

From the proof of Lemma \ref{NC1}, $\Phi_+=\Theta\Psi(0)^*$ and $\Phi_-=\Theta(0)\Psi^*$ for operator valued inner functions $\Theta$ and $\Psi$. Observe that
\[
(\Theta(e^{it}) \psi_0^*)^*(\Theta(e^{it}) \psi_0^*)=\psi_0\psi_0^* \,\, \quad(\text{a.e.\ on } \T).
\]
Thus we get $\Phi_+(e^{it})^*\Phi_+(e^{it}) $ is a constant positive operator a.e.\ on $\T$. Similarly,  $\Phi_-(e^{it})\Phi_-(e^{it})^*=(\theta_0 \Psi(e^{it})^*)(\theta_0 \Psi(e^{it})^*)^* = \theta_0 \theta_0^*$ is a constant positive operator a.e.\ on 
$\T$. 

Now let $\phi_{-m_0}$ be nonzero for some $m_0 \geq 1$. For $\eta \in \cle$, define a function $f\in H^2_{\cle_*}(\D)$ by 
\[
f(z)=(T_\Phi M_z-M_zT_\Phi)\left(\eta z^{m_0-1}\right).
\]
Then 
\begin{align*}
f(e^{it})&=T_\Phi(\eta e^{im_0t})-e^{it}T_\Phi(\eta e^{i(m_0-1)t})\\
&=P_+\left(\sum\limits_{m=1}^\infty \phi_{-m}(\eta)e^{i(-m+m_0)t}\right)+\sum\limits_{m=0}^\infty \phi_m(\eta)e^{i(m+m_0)t} \\
& \quad -e^{it}P_+\left(\sum\limits_{m=1}^\infty \phi_{-m}(\eta)e^{i(-m+m_0-1)t}\right) - e^{it}\sum\limits_{m=0}^\infty \phi_{m}(\eta)e^{i(m+m_0-1)t}\\
&=\sum\limits_{m=1}^{m_0} \phi_{-m}(\eta)e^{i(-m+m_0)t}-e^{it}\left(\sum\limits_{m=1}^{m_0-1}\phi_{-m}(\eta)e^{i(-m+m_0-1)t}\right)\\
&=\phi_{-m_0}(\eta).
\end{align*}
Since $T_\Phi$ is a partial isometry, $T_\Phi^*$  is also a partial isometry. 
Hence $[\cln(T_\Phi^*)]^\perp=\clr(T_\Phi)$ is $M_z$-invariant which yields $f \in \clr(T_\Phi)$.
Also note that 
\[
\|f\|=\|T_\Phi^*f\|=\|P_+L_\Phi ^*f\|\leq \|L_\Phi^* f \|\leq \|f\|.
\]
Therefore, $\|P_+L_\Phi ^*f \|=\|L_\Phi^* f \|$ which implies $P_+ L_\Phi^* (f)=L_\Phi^*f$. Hence $L_\Phi^* f \in H^2_\cle(\D)$. Then
\[
\sum\limits_{m=1}^\infty \phi_{-m}^*\phi_{-m_0}(\eta)e^{imt}+\sum\limits_{m=0}^\infty \phi_m^*\phi_{-m_0}(\eta)e^{-imt} = \Phi(e^{it})^*f(e^{it})\in H^2_\cle(\D)
\]
if and only if 
\[
\phi_m^*\phi_{-m_0}\eta=0 \quad \mbox{for all} ~~ m \geq 1.
\]
Since $m_0 \geq 1$ is arbitrary such that $\phi_{-m_0} \neq 0$, we have $\phi_m^*\phi_{-n}=0 $ for all $m,n \geq 1$. Note that if $\phi_{-m}=0$ for all $m \geq 1$, then this condition is trivially true. 

Again, since $T_\Phi^*$ is also a partial isometry, on replacing $\Phi$ by $\Phi^*$, we obtain $\phi_{-n}\phi_m^*=0$ for all $m ,n \geq 1$. This finishes the proof.
\end{proof}

Now it is a natural question to ask whether the converse of Proposition \ref{NC} holds. We give an affirmative answer in the case when $\cle$ and $\cle_*$ are finite-dimensional. Before that, a few general observations are nevertheless in order.
\begin{rem}\label{innerproduct}
For finite-dimensional Hilbert spaces $\cle$ and  $\cle_*$, $\clb(\cle,\cle_*)$ can be viewed as a Hilbert space (with Hilbert-Schmidt norm). And thus we may view $L^2_{\clb(\cle,\cle_*)}$ as a space of square summable Fourier series with coefficients in $\clb(\cle,\cle_*)$ with the following inner product (see \cite{GHKL}).
\[
\langle A, B\rangle=\int_{\T} tr(B^*A) d\mu=\sum\limits_{n=-\infty}^\infty tr(B_n^*A_n),
\]
where $A(e^{it})=\sum\limits_{n=-\infty}^\infty A_ne^{int}$, $B(e^{it})=\sum\limits_{n=-\infty}^\infty B_ne^{int} \in L^2_{\clb(\cle,\cle_*)}$ with $A_n, B_n \in \clb(\cle,\cle_*)$ for all $n$. In particular, $L^\infty_{\clb(\cle,\cle_*)} \subset L^2_{\clb(\cle,\cle_*)}.$ Note that since $\clb(\cle,\cle_*)$ is finite-dimensional, operator norm and Hilbert-Schmidt norm are equivalent. So, we continue to consider on $L^\infty _{\clb(\cle,\cle_*)}$ and $H^\infty_{\clb(\cle,\cle_*)}$, the equivalent norm obtained by operator norm.
\end{rem}

\begin{prop}\label{SC}
Let $\cle,\cle_*$ be finite-dimensional Hilbert spaces and  $\Phi\in \cla$ be a nonzero partial isometry a.e.\ on $\T$ which satisfies $(\ref{con1})$ and $(\ref{con2})$ of Proposition \ref{NC}. Then $T_\Phi$ is a partially isometric Toeplitz operator from $H^2_\cle(\D)$ to $ H^2_{\cle_*}(\D)$.
\end{prop}

\begin{proof}
Suppose that $\Phi \in \cla$ is a partial isometry a.e.\ on $\T$ such that $(\ref{con1})$ and $(\ref{con2})$ of Proposition $\ref{NC}$ holds where $\Phi_+(z)=\sum\limits_{m=0}^\infty \phi_mz^{m} \in H^\infty _{\clb(\cle,\cle_*)}$ and $\Phi_-(z)=\sum\limits_{m=0}^\infty \phi_{-m}\bar{z}^m \in \bar H^\infty_{\clb(\cle,\cle_*)}$.\\
Let $\Phi_+(e^{it})^*\Phi_+(e^{it})=T$ and $\Phi_-(e^{it})\Phi_-(e^{it})^*=V$ a.e.\ on $\T$. Then $T$ and $V$ are positive operators.
For $\eta \in \cle$, consider
\begin{align}
	\| \Phi_+\eta \|^2_{H^2_{\cle_*}(\D)}&=\frac{1}{2\pi}\int_{0}^{2\pi}\|\Phi_+(e^{it})\eta\|^2_{\cle_*} dt \nonumber\\
	&=\frac{1}{2\pi}\int_0^{2\pi}\left\langle \Phi_+(e^{it})\eta,\Phi_+(e^{it})\eta\right\rangle_{\cle_*} dt \nonumber\\
	&=\frac{1}{2\pi}\int_0^{2\pi}\left\langle \Phi_+(e^{it})^*\Phi_+(e^{it})\eta,\eta\right\rangle_{\cle} dt \nonumber\\
	&=\langle T\eta,\eta\rangle. \label{Phi1}
\end{align}
Also,
\begin{equation} \label{Phi2}
	\| \Phi_+\eta \|^2_{H^2_{\cle_*}(\D)}=\sum\limits_{m=0}^\infty \|\phi_m \eta\|^2_{\cle_*}=\left\langle \sum\limits_{m=0}^\infty \phi_m^*\phi_m\eta,\eta \right\rangle_{\cle} .
\end{equation}
Note that the infinite series $\sum\limits_{m=0}^\infty \phi_m^*\phi_m$ is  absolutely convergent since $\Phi_+ \in L^\infty_{\clb(\cle,\cle_*)} \subset L^2_{\clb(\cle,\cle_*)}$ (see Remark \ref{innerproduct}). Thus from (\ref{Phi1}) and (\ref{Phi2}), we get
\begin{equation}\label{Tdef}
T\eta =\sum\limits_{m=0}^\infty \phi_m^*\phi_m\eta \quad (\eta \in \cle) .
\end{equation}
Now for $\eta \in \cle$, consider 
\begin{align*}
L_\Phi L_\Phi^*L_\Phi(\eta) &=L_\Phi L_\Phi^*(\Phi_++\Phi_--\phi_0)\eta\\
 &=L_\Phi((\Phi_+-\phi_0)^*\Phi_+\eta+(\Phi_+-\phi_0)^*(\Phi_--\phi_0)\eta+\Phi_-^*\Phi_+\eta+\Phi_-^*(\Phi_--\phi_0)\eta)\\
 &\stackrel{(\ref{con2})}{=}L_\Phi(T-\phi_0^*\Phi_+\eta+\Phi_-^*\phi_0\eta+\phi_0^*(\Phi_+-\phi_0)\eta+\Phi_-^*(\Phi_--\phi_0)\eta)\\
 &= L_\Phi(T-\phi_0^*\phi_0+\Phi_-^*\Phi_-)\eta\\
 &=((\Phi_+-\phi_0)T-(\Phi_+-\phi_0)\phi^*_0\phi_0+(\Phi_+-\phi_0)\Phi_-^*\Phi_-+\Phi_-T-\Phi_-\phi_0^*\phi_0+\Phi_-\Phi_-^*\Phi_-)\eta\\
 &\stackrel{(\ref{con})}{=} ((\Phi_+-\phi_0)T-(\Phi_+-\phi_0)\phi_0^*\phi_0+(\Phi_+-\phi_0)\phi_0^*\Phi_-+\Phi_-T-\Phi_-\phi_0^*\phi_0+V\Phi_-)\eta. 
\end{align*}
Since $\Phi(e^{it})$ is a partial isometry a.e.\ on $\T$, $L_\Phi L_\Phi^*L_\Phi=L_\Phi$. Thus
\begin{equation}\label{p.i}
((\Phi_+-\phi_0)T+(\Phi_+-\phi_0)\phi_0^*(\Phi_--\phi_0)+\Phi_-T-\Phi_-\phi_0^*\phi_0+V\Phi_-)\eta=(\Phi_++\Phi_--\phi_0)\eta.
\end{equation}
Now let $j \geq 1$, then applying $L_{\phi_{-j}}^*$ on both sides of (\ref{p.i}) and using condition $(\ref{con2})$, we get
\begin{equation}\label{con}
(\phi_{-j}^*\Phi_-T-\phi_{-j}^*\Phi_-\phi_0^*\phi_0+\phi_{-j}^*V\Phi_-)\eta =\phi_{-j}^*\Phi_-(\eta).
\end{equation}
Comparing the $(-j)^{th}$ Fourier coefficient on both sides, we get
\[
(\phi_{-j}^*\phi_{-j}T-\phi_{-j}^*\phi_{-j}\phi_0^*\phi_0+\phi_{-j}^*V\phi_{-j})\eta=\phi_{-j}^*\phi_{-j}\eta.
\]
Since it is true for all $\eta \in \cle$, we get
\begin{align*}
&\phi_{-j}^*\phi_{-j}-\phi_{-j}^*\phi_{-j}(T-\phi_0^*\phi_0)-\phi_{-j}^*V\phi_{-j}=0\\
\implies&\phi_{-j}^*\phi_{-j}-\phi_{-j}^*\phi_{-j}\left(\sum\limits_{m=1}^\infty \phi_m^*\phi_m\right)-\phi_{-j}^*V\phi_{-j}=0 \quad \text{(By (\ref{Tdef}))}\\ 
\implies & \phi_{-j}^*\phi_{-j}-\phi_{-j}^*V\phi_{-j}=0.
\end{align*}
As $I-V\geq 0$, we have
\[
(I -V)\phi_{-j}=0, 
\]
i.e.,
\begin{equation} \label{con3}
V\phi_{-j}=\phi_{-j} \quad \forall \,j \geq 1.
\end{equation}
Similarly, one can prove (by applying $L_{\phi_j}$ on both sides of (\ref{p.i})  and comparing $j^{th}$ coefficient) that
\begin{equation}\label{con4}
\phi_jT=\phi_j \quad \forall \, j \geq 1.
\end{equation}
Also, comparing constant term of (\ref{con}), we get
\begin{align*}
&(\phi_{-j}^*\phi_0T-\phi_{-j}^*\phi_0\phi_0^*\phi_0+\phi_{-j}^*V\phi_0)\eta=\phi_{-j}^*\phi_0(\eta)\\
\implies 	&\phi_{-j}^*\phi_0(T-\phi_0^*\phi_0)=0 \quad \quad (\text{By (\ref{con3})})\\
\implies &\phi_{-j}^*\phi_0\left(\sum\limits_{m=1}^\infty \phi_m^*\phi_m\right)=0 \quad\quad (\text{By (\ref{Tdef})})\\
\implies & \sum\limits_{m=1}^\infty (\phi_m\phi_0^*\phi_{-j})^*(\phi_m\phi_0^*\phi_{-j})=0.
\end{align*}
Therefore, 
\begin{equation}\label{con5}
\phi_{m}\phi_0^*\phi_{-j}=0 \,\, ~~ \forall \,\,j,m \geq 1.
\end{equation}
In that case, using (\ref{con3}), (\ref{con4}), and (\ref{con5}), equation (\ref{p.i}) will become
\begin{align}
	&((\Phi_+-\phi_0)+\Phi_-T-\Phi_-\phi_0^*\phi_0+V\phi_0+(\Phi_--\phi_0))\eta=(\Phi_++\Phi_--\phi_0)\eta \nonumber\\ 
	\implies & (\Phi_-(T-\phi_0^*\phi_0)+V\phi_0)\eta=\phi_0\eta. \label{con6}
\end{align}
Define 
\[
T_\Phi: = P_{+}^{\cle_{*}}L_{\Phi}|_{ H^2_\cle(\D)}.
\]
Now for $\eta \in \cle$ and $j \geq 0$,
\begin{align*}
&T_\Phi T_\Phi^*T_\Phi(\eta e^{ijt})\\
&=T_\Phi T_\Phi^*\left(\Phi_+(\eta)e^{ijt}+\sum\limits_{m=0}^j \phi_{-m}(\eta)e^{i(j-m)t}-\phi_0e^{ijt}\right)\\
&=T_\Phi P_+^{\cle}\left(\Phi_+^*(\Phi_+-\phi_0)(\eta e^{ijt})+(\Phi_--\phi_0)^*(\Phi_+-\phi_0)\eta e^{ijt}+\Phi_+^*\left(\sum\limits_{m=0}^j \phi_{-m}(\eta) e^{i(j-m)t}\right)+\right.\\
&\quad\quad\quad\quad\quad\quad\quad\quad\quad\left.(\Phi_--\phi_0)^*\left(\sum\limits_{m=0}^j \phi_{-m}(\eta)e^{i(j-m)t}\right)\right)\\
&\stackrel{(\ref{con2})}{=}T_\Phi P_+^{\cle} \left(T\eta e^{ijt}-\Phi_+^*\phi_0\eta e^{ijt}+\phi_0^*\left(\sum\limits_{m=0}^j \phi_{-m}(\eta)e^{i(j-m)t}\right)+(\Phi_+-\phi_0)^*\phi_0\eta e^{ijt}+\right.\\
&\quad\quad\quad\quad\quad\quad\quad\quad\quad\left.(\Phi_--\phi_0)^*\left(\sum\limits_{m=0}^j \phi_{-m}(\eta)e^{i(j-m)t}\right)\right)\\
&=T_\Phi\left(T\eta e^{ijt}-\phi_0^*\phi_0 \eta e^{ijt}+\Phi_-^*\left(\sum\limits_{m=0}^j \phi_{-m}(\eta)e^{i(j-m)t}\right)\right)\\
&=P_+^{\cle_*}\left((\Phi_+-\phi_0)T\eta e^{ijt}+\Phi_-T\eta e^{ijt}-(\Phi_+-\phi_0)\phi_0^*\phi_0\eta e^{ijt}-\Phi_-\phi_0^*\phi_0\eta e^{ijt}+\right.\\
&\quad\quad\quad\quad\quad\quad\quad\quad\quad\left.(\Phi_+-\phi_0)\Phi_-^*\left(\sum\limits_{m=0}^j\phi_{-m}(\eta)e^{i(j-m)t}\right)+\Phi_-\Phi_-^*\left(\sum\limits_{m=0}^j\phi_{-m}(\eta)e^{i(j-m)t}\right)\right)\\
&\stackrel{(\ref{con4}),(\ref{con2})}{=}P_+^{\cle_*}\left((\Phi_+-\phi_0)\eta e^{ijt}+\Phi_-T\eta e^{ijt}-(\Phi_+-\phi_0)\phi_0^*\phi_0\eta e^{ijt}-\Phi_-\phi_0^*\phi_0\eta e^{ijt}+\right.\\
&\quad\quad\quad\quad\quad\quad\quad\quad\quad\left.(\Phi_+-\phi_0)\phi_0^*\left(\sum\limits_{m=0}^j\phi_{-m}(\eta)e^{i(j-m)t}\right)+V\left(\sum\limits_{m=0}^j\phi_{-m}(\eta)e^{i(j-m)t}\right)\right)\\
&\stackrel{(\ref{con5})}{=}P_+^{\cle_*}\left((\Phi_+-\phi_0)\eta e^{ijt}+\Phi_-T\eta e^{ijt}-(\Phi_+-\phi_0)\phi_0^*\phi_0\eta e^{ijt}-\Phi_-\phi_0^*\phi_0\eta e^{ijt}+\right.\\
&\quad\quad\quad\quad\quad\quad\quad\quad\quad\left.(\Phi_+-\phi_0)\phi_0^*\phi_0\eta e^{ijt}+V\left(\sum\limits_{m=0}^j\phi_{-m}(\eta)e^{i(j-m)t}\right)\right)\\
&\stackrel{(\ref{con3})}{=}\begin{cases}
	 P_+^{\cle_*}\left((\Phi_+-\phi_0)\eta e^{ijt}+\Phi_-(T-\phi_0^*\phi_0)\eta e^{ijt} + V\phi_0\eta e^{ijt}+\sum\limits_{m=1}^j\phi_{-m}(\eta)e^{i(j-m)t}\right),& j>0\\
	 P_+^{\cle_*}\left((\Phi_+-\phi_0)\eta e^{ijt}+\Phi_-(T-\phi_0^*\phi_0)\eta e^{ijt} + V\phi_0\eta e^{ijt}\right), & j = 0 
\end{cases} \\
&\stackrel{(\ref{con6})}{=}\begin{cases}
	 (\Phi_+-\phi_0)\eta e^{ijt}+ \phi_0\eta e^{ijt}+\sum\limits_{m=1}^j\phi_{-m}(\eta)e^{i(j-m)t}, & j>0\\
	 (\Phi_+-\phi_0)\eta e^{ijt}+ \phi_0\eta e^{ijt}, & j =0
	\end{cases}\\
&=T_\Phi(\eta e^{ijt}).
\end{align*}
Thus $T_\Phi$ is a partially isometric Toeplitz operator. This completes the proof.
\end{proof}

Combining Propositions $\ref{NC}$ and $\ref{SC}$, we get the following result.

\begin{thm}\label{Characterization-Partial}
For finite-dimensional Hilbert spaces $\cle$ and $\cle_*$, let $\Phi \in \cla$ be such that $\Phi(e^{it})$ is a nonzero partial isometry a.e.\ on $\T$. Then $T_\Phi$ is a partially isometric Toeplitz operator if and only if the following conditions are satisfied:
\begin{enumerate}
\item  $\Phi_+(e^{it})^*\Phi_+(e^{it}) $ and $\Phi_-(e^{it})\Phi_-(e^{it})^*$
are  operator-valued constant functions a.e.\ on $\T$, where $\Phi_+ \in H^\infty_{\clb(\cle,\cle_*)}$ and $\Phi_-^* \in H^\infty_{\clb(\cle_*,\cle)}$ are contractive analytic functions.
\item  $\hat{\Phi}(m)^*\hat{\Phi}(-n)=0_\cle$ and $\hat{\Phi}(-n)\hat{\Phi}(m)^*=0_{\cle_*}$ for all $m,n \geq 1$.
\end{enumerate}
\end{thm}

We conclude this section with the following remark.
\begin{rem}
In particular, for $\cle=\cle_*=\C$ (scalar-valued Hardy space), the above theorem says that for $0\neq \Phi \in L^\infty$ such that $\Phi(e^{it})\overline{\Phi(e^{it})}\Phi(e^{it})=\Phi(e^{it})$ a.e.\ on $\T$, $T_\Phi \in \clb(H^2(\D))$ is a partial isometry if and only if $\overline{\phi}_{-n}\phi_m=0$ for all $m,n \geq 1$, i.e., either all negative Fourier coefficients are zero or all positive coefficients are zero. Hence, $T_\Phi$ is either an isometry or a co-isometry which was first proved by Brown and Douglas in \cite{BD-Partially}.
\end{rem}

\section{Examples}{\label{sec5}}
In this final section we shall illustrate some examples that none of the  conditions of Theorem $\ref{Characterization-Partial}$ is redundant.

\begin{ex}
Let $\Phi \in L^\infty_{\clb(\C^3)}$ be defined by
\[
\Phi(z)=\begin{pmatrix}
0&\frac{z}{2} & \frac{\sqrt{3}}{2} \\
0 &0 &0\\
\bar{z} & 0 & 0
\end{pmatrix}.
\]
Then
\[
\Phi(z)=\phi_{-1}\bar{z} + \phi_{0} + \phi_{1}z = 
\begin{pmatrix}
0&0 & 0 \\
0 &0 &0\\
1 & 0 & 0
\end{pmatrix}\bar{z}
+  
\begin{pmatrix}
0& 0 & \frac{\sqrt{3}}{2} \\
0 &0 &0\\
0 & 0 & 0
\end{pmatrix}
+ 
\begin{pmatrix}
0&\frac{1}{2} & 0 \\
0 &0 &0\\
0 & 0 & 0
\end{pmatrix}z.
\]
Clearly, $\Phi(e^{it})\Phi(e^{it})^*\Phi(e^{it})=\Phi(e^{it})$ and hence, 
$\Phi(e^{it})$ is a partial isometry a.e.\ on $\T$.
Here,
\[
\Phi_+(e^{it})=\begin{pmatrix}
0&\frac{e^{it}}{2} & \frac{\sqrt{3}}{2}\\
0 &0 &0\\
0 &0 & 0
\end{pmatrix},
\Phi_-(e^{it})=\begin{pmatrix}
 0 &0& \frac{\sqrt{3}}{2}\\
0 &0 &0\\
e^{-it} &0 &0
\end{pmatrix}.
\]
Also
\[
\phi_1^*\phi_{-1}= \begin{pmatrix}
0&0&0\\
\frac{1}{2} &0&0\\
0&0&0
\end{pmatrix}
\begin{pmatrix}
0 & 0& 0\\
0&0&0\\
1&0&0
\end{pmatrix}=0.
\]
Similarly, one can prove that $\phi_{-1}\phi_{1}^*=0$. Also,
\[
\Phi_-(e^{it})\Phi_-(e^{it})^*=
\begin{pmatrix}
\frac{3}{4} & 0& 0\\
0&0&0\\
0 &0&1
\end{pmatrix}.
\]
But
\[
\Phi_+(e^{it})^*\Phi_+(e^{it})=
\begin{pmatrix}
0 & 0& 0\\
0&\frac{1}{4}&\frac{\sqrt{3}}{4}e^{-it}\\
0&\frac{\sqrt{3}}{4}e^{it} &\frac{3}{4}
\end{pmatrix}
\]
is not operator-valued constant function. It is easy to see that
\[T_\Phi=
\begin{pmatrix}
O &\frac{1}{2} T_{e^{it}}& \frac{\sqrt 3}{2}I\\
O & O & O\\
T_{e^{-it}} & O &O
\end{pmatrix}
\]
is not a partial isometry. 
\end{ex}

\begin{ex}
Define $\Phi \in L^\infty_{\clb(\C^3)}$ as
\[
\Phi(z)=\begin{pmatrix}
0&\frac{\bar{z}}{2} & 0 \\
0 &\frac{\sqrt{3}}{2} &0\\
z & 0 & 0
\end{pmatrix}.
\]
It is easy to check that $\Phi(e^{it})$ is a partial isometry a.e.\ on $\T$.
Here,
\[
\Phi_+(e^{it})=\begin{pmatrix}
0& 0 & 0\\
0 &\frac{\sqrt{3}}{2} &0\\
e^{it} &0 & 0
\end{pmatrix},
\Phi_-(e^{it})=\begin{pmatrix}
 0 &\frac{e^{-it}}{2}&0\\
0 &\frac{\sqrt{3}}{2} &0\\
0 &0 &0 
\end{pmatrix}
\]
We can check that $\phi_1^*\phi_{-1}=0$ and $\phi_{-1}\phi_1^*=0.$
Also note that
\[
\Phi_+(e^{it})^*\Phi_+(e^{it})=
\begin{pmatrix}
1 & 0& 0\\
0&\frac{3}{4}&0\\
0&0&0
\end{pmatrix}
\]
which is constant. But
\[
\Phi_-(e^{it})\Phi_-(e^{-it})^*=
\begin{pmatrix}
\frac{1}{4} & \frac{\sqrt{3}}{4}e^{-it}& 0\\
\frac{\sqrt{3}}{4}e^{it}&\frac{3}{4}&0\\
0&0&0
\end{pmatrix}
\]
is not a constant. Clearly, $T_\Phi$ is not a partial isometry. Hence the condition $\Phi_-(e^{it})\Phi_-(e^{it})^*$ is an operator-valued constant function cannot be dropped. 
\end{ex}

\begin{ex}
Let $\Phi \in L^\infty_{\clb(\C^2)}$ be defined by
\[
\Phi(z)=\begin{pmatrix}
\frac{z}{\sqrt2} &\frac{\bar{z}}{\sqrt{2}}\\
0&0
\end{pmatrix}.
\]
It is trivial to check that $\Phi(e^{it})$ is a partial isometry a.e.\ on $\T$ and  $\Phi_+(e^{it})^*\Phi_+(e^{it})$ and $\Phi_-(e^{it})\Phi_-(e^{it})^*$ are operator-valued constant functions. Also
\[
\phi_{-1}\phi_{1}^*=\begin{pmatrix}
0 & \frac{1}{\sqrt{2}}\\
0 &0
\end{pmatrix}
\begin{pmatrix}
\frac{1}{\sqrt{2}} & 0\\
0 &0
\end{pmatrix}=0.
\]
But
\[
\phi_1^*\phi_{-1}=\begin{pmatrix}
\frac{1}{\sqrt{2}} &0 \\
0&0
\end{pmatrix}
\begin{pmatrix}
0 &\frac{1}{\sqrt{2}}\\
0 &0
\end{pmatrix}=\begin{pmatrix}
0&\frac{1}{2}\\
0 &0
\end{pmatrix} \neq 0.
\]
It is a routine check to see that $T_\Phi$ is not a partial isometry. Thus the condition $\phi_n^*\phi_{-m}=0$ for all $m, n \geq 1$ cannot be removed. 
\end{ex}

\begin{ex}
Let $\Phi \in L^\infty_{\clb(\C^2)}$ be a partial isometric symbol defined by
\[
\Phi(z)=\begin{pmatrix}
\frac{\bar{z}}{\sqrt{2}}& 0\\
\frac{z}{\sqrt2} &0
\end{pmatrix}.
\]
It is trivial to check that $\Phi_+(e^{it})^*\Phi_+(e^{it})$ and $\Phi_-(e^{it})\Phi_-(e^{it})^*$ are constant. Also
\[
\phi_{1}^*\phi_{-1}=\begin{pmatrix}
0 & \frac{1}{\sqrt{2}}\\
0 &0
\end{pmatrix}
\begin{pmatrix}
\frac{1}{\sqrt{2}} & 0\\
0 &0
\end{pmatrix}=0.
\]
But
\[
\phi_{-1}\phi_{1}^*=\begin{pmatrix}
\frac{1}{\sqrt{2}} & 0\\
0&0
\end{pmatrix}
\begin{pmatrix}
0 &\frac{1}{\sqrt{2}}\\
0 &0
\end{pmatrix}=\begin{pmatrix}
0&\frac{1}{2}\\
0 &0
\end{pmatrix} \neq 0.
\]
One can check easily that $T_\Phi$ is not a partial isometry. 
Hence $\phi_{-m}\phi_{n}^*=0$ for all $m, n \geq 1$ cannot be dismissed.
\end{ex}

\NI\textit{Acknowledgements:} 
The second author's research work is partially supported by the Core Research Grant (CRG/2022/006891), SERB (DST), Government of India. 
\vspace{0.4cm}

\NI\textit{Data availability:}
Data sharing is not applicable to this article as no data sets
were generated or analysed during the current study.
\vspace{0.2cm}

\NI\textit{Declarations}
\vspace{0.2cm}

\NI\textit{Conflict of interest:}
The authors have no competing interests to declare.

\label{Ref}

\end{document}